\documentclass{article}

\newif\ifdraft \drafttrue  

\usepackage{color}
\usepackage{graphicx}
\RequirePackage[pdftex,pdfpagemode=none, pdftoolbar=true,
                pdffitwindow=true,pdfcenterwindow=true]{hyperref}

\def\Diamond{*}

\newcount\firstpage
\ifdraft \firstpage=113 \else
\edef\temp{(\noexpand\firstpage is set to \the\firstpage.)}
\immediate\write16{*Note*: Set \firstpage to the 1st page number for
  correct page references.}
\immediate\write16\expandafter{\temp}
\fi

\makeatletter 

\newtheorem{THEOREM}{Theorem}[section]
\newtheorem{PROPOSITION}[THEOREM]{Proposition}
\newtheorem{DEFINITION}[THEOREM]{Definition}
\newtheorem{LEMMA}[THEOREM]{Lemma}
\newtheorem{COROLLARY}[THEOREM]{Corollary}
\newtheorem{EXAMPLE}[THEOREM]{Example}
\newtheorem{REMARK}[THEOREM]{Remark}
\newtheorem{REMARKS}[THEOREM]{Remarks}
\newtheorem{DISCUSSION}[THEOREM]{}
\newtheorem{ALGORITHM}[THEOREM]{Algorithm}
\newtheorem{PROBLEM}[THEOREM]{Problem}

\def\thmstretch{plus.5em minus.1em}
\def\thmskip{0pt \thmstretch}
\def\thmstart{\hskip\thmskip\ignorespaces}
\let\thmtextfont=\it
\newif\ifthmtitle \thmtitletrue
\newenvironment{theorem}%
  {\let\thmtextfont=\it \thmtitletrue \begin{THEOREM}}%
{\end{THEOREM}}
\newenvironment{proposition}%
  {\let\thmtextfont=\it \thmtitletrue \begin{PROPOSITION}}%
  {\end{PROPOSITION}}
\newenvironment{definition}%
  {\let\thmtextfont=\rm \thmtitletrue \begin{DEFINITION}}%
  {\end{DEFINITION}}
  {\let\thmtextfont=\it \thmtitletrue \begin{LEMMA}}%
  {\end{LEMMA}}
\newenvironment{corollary}%
  {\let\thmtextfont=\it \thmtitletrue \begin{COROLLARY}}%
  {\end{COROLLARY}}
\newenvironment{example}%
  {\let\thmtextfont=\rm \thmtitletrue \begin{EXAMPLE}}%
  {\end{EXAMPLE}}
\newenvironment{remark}%
  {\let\thmtextfont=\rm \thmtitletrue \begin{REMARK}}%
  {\end{REMARK}}
  {\let\thmtextfont=\rm \thmtitletrue \begin{REMARKS}}%
  {\end{REMARKS}}
  {\let\thmtextfont=\rm \thmtitlefalse \begin{DISCUSSION}}%
  {\end{DISCUSSION}}
\newenvironment{algorithm}%
  {\let\thmtextfont=\rm \thmtitletrue \begin{ALGORITHM}}%
  {\end{ALGORITHM}}
\newenvironment{problem}%
  {\let\thmtextfont=\rm \thmtitletrue \begin{PROBLEM}}%
  {\end{PROBLEM}}

\newenvironment{steps}{\begin{description}}{\end{description}}
\def\step[#1]{\item[\rm{\it Step\/}~#1]}

\def\@begintheorem#1#2{\thmtextfont \trivlist \ifthmtitle
  \item[\hskip \labelsep{\bf #1\ #2.}]\else
  \item[\hskip \labelsep{\bf #2.}]\fi\thmstart}
\def\@opargbegintheorem#1#2#3{\thmtextfont \trivlist
      \item[\hskip\labelsep{\bf #1\ #2\ {\rm(#3)}.}]\thmstart}

\def\specialsection#1{\noindent
  {\bf#1.}\hskip\labelsep\thmstretch\ignorespaces}

\renewenvironment{abstract}%
  {\par\addvspace{\bigskipamount}\small\specialsection{Abstract}}%
  {\par}

\newenvironment{keywords}%
  {\par\addvspace{\medskipamount}\small\it\specialsection{Keywords}}%
  {\tolerance=300\par}

\newenvironment{acknowledgments}%
  {\par\addvspace{\bigskipamount}\specialsection{Acknowledgments}}%
  {\tolerance=300\par}

\def\affiliation#1{\date{\normalsize\it #1}}

\newif\ifmathtomb \mathtombfalse

\def\tombstone{\unskip\penalty50   
  \hskip 0pt plus-1fill \null\nobreak\hskip 0pt plus1fill
  \enskip \vrule width.3333em height.7em depth.2em
  \ifmmode \global\mathtombtrue \else \global\mathtombfalse \fi}

\newenvironment{proof}
  {\futurelet\next\pr@oftext}
  {\ifmathtomb \else \tombstone \fi \widowpenalty=10000  
   \par \ifmathtomb \else \addvspace{\medskipamount}\fi \global\mathtombfalse}
\def\pr@oftext{\ifx\next[\let\temp\opr@@ftext\else\let\temp\pr@@ftext\fi\temp}
\def\pr@@ftext{\beginpr@@f{Proof}}
\def\opr@@ftext[#1]{\beginpr@@f{#1}}
\def\beginpr@@f#1{\par \addvspace{\bigskipamount}%
  \noindent{\bf #1.}\hskip\labelsep\thmstretch\ignorespaces}

\def\@citex[#1]#2{\if@filesw\immediate\write\@auxout{\string\citation{#2}}\fi
  \def\@citea{}\@cite{\@for\@citeb:=#2\do
    {\@citea\def\@citea{,\penalty\@m\thinspace }\@ifundefined
       {b@\@citeb}{{\bf ?}\@warning
       {Citation `\@citeb' on page \thepage \space undefined}}%
\hbox{\bf\csname b@\@citeb\endcsname}}}{#1}}

\newtoks\headline \headline={\hfil} 
\newtoks\footline \footline={\hss\rm\thepage\hss}

\def\ps@plain{\let\@mkboth\@gobbletwo
  \def\@oddhead{\the\headline}\let\@evenhead\@oddhead
  \def\@oddfoot{\the\footline}\let\@evenfoot\@oddfoot}

\def\runningtitle{Optimization on Riemannian Manifolds}  
\def\runningname{Steven T. Smith}
\def\copyrightinfo{\copyright\ 1993 by the American Mathematical Society.}

\headline={\ifnum\c@page=\firstpage\hfil\else
  \ifodd\c@page \hfil{\footnotesize\rm
    \uppercase\expandafter{\runningtitle}}\hfil\llap{\rm\thepage}\else
    \rlap{\rm\thepage}\hfil{\footnotesize\rm
      \uppercase\expandafter{\runningname}}\hfil\fi\fi}
\footline={\ifnum\c@page=\firstpage
  {\footnotesize\rm\copyrightinfo}\hfil\llap{\rm\thepage}\else
    \hfil\fi}

\ifdraft
\def\publishinginfo{\vtop{\footnotesize\hbox{Fields Institute Communications}
  \hbox{Volume {\bf3}, 1994}}}
\def\parbox#1{\setbox1=\vtop{\small\rm\pretolerance=9999
  \parskip=0pt\noindent\ignorespaces #1\par}\dp1=0pt\box1}

\headline={\ifnum\c@page=\firstpage\hss\parbox\publishinginfo\hss\else
  \ifodd\c@page \hfil{\footnotesize\rm
    \uppercase\expandafter{\runningtitle}}\hfil\llap{\rm\thepage}\else
    \rlap{\rm\thepage}\hfil{\footnotesize\rm
      \uppercase\expandafter{\runningname}}\hfil\fi\fi}
\footline={\ifnum\c@page=\firstpage
  {\footnotesize\rm\copyrightinfo}\hfil\llap{\rm\thepage}\else
    \hfil\fi}
\fi

\def\thinskip{\hskip .16667em }

{\catcode`|=\active \catcode`\!=\active
\gdef\references{\catcode`|=\active \catcode`\!=\active
  \def\!{\char`\!}%
  \def|{\bgroup\sc \let|=\egroup}     
  \def!{\bgroup\it \let!=\egroup}     
  \def\<##1>{{\bf##1}\futurelet\next\number@ptional}  
  \def\number@ptional{\ifx\next(\def\@temp{\n@mber}\else
    \ifx\next:\def\@temp{\p@ges}\else\def\@temp{}\fi\fi \@temp}%
  \def\n@mber(##1){\thinspace(##1)\futurelet\next\page@ptional}%
  \def\page@ptional{\ifx\next:\def\@temp{\p@ges}\else
    \def\@temp{}\fi \@temp}%
  \def\p@ges:{\thinspace:\penalty-20\thinskip}%
  \frenchspacing}}
\def\authorbar{$\vcenter{\hbox{\vrule width3em height.4pt}}\,$}

\def\gnulog 1e{\futurelet\next\gnul@g}
\def\gnul@g{\ifx\next+\let\temp\gnul@gp\else\let\temp\gnul@gm\fi\temp}
\def\gnul@gp+{\afterassignment\gnul@gg\count10=}
\def\gnul@gm-{\afterassignment\gnul@gg-\count10=}
\def\gnul@gg{\the\count10}

\def\eqnarray{\stepcounter{equation}\let\@currentlabel=\theequation
  \global\@eqnswtrue
  \global\@eqcnt\z@\tabskip\@centering\let\\=\@eqncr
  $$\halign to \displaywidth\bgroup\@eqnsel\hskip\@centering
    $\displaystyle\tabskip\z@{##}$&\global\@eqcnt\@ne 
    ${{}##{}}$&\global\@eqcnt\tw@$\displaystyle\tabskip\z@{##}$\hfil
    \tabskip\@centering&\llap{##}\tabskip\z@\cr}

\font\eightex=cmex8
\font\sevenbf=cmbx7
\font\fivebf=cmbx5

\def\xpt{\textfont\z@\tenrm
  \scriptfont\z@\sevrm \scriptscriptfont\z@\fivrm
\textfont\@ne\tenmi \scriptfont\@ne\sevmi \scriptscriptfont\@ne\fivmi
\textfont\tw@\tensy \scriptfont\tw@\sevsy \scriptscriptfont\tw@\fivsy
\textfont\thr@@\tenex \scriptfont\thr@@\eightex \scriptscriptfont\thr@@\eightex
\def\unboldmath{\everymath{}\everydisplay{}\@nomath\unboldmath
          \textfont\@ne\tenmi
          \textfont\tw@\tensy \textfont\lyfam\tenly
          \@boldfalse}\@boldfalse
\def\boldmath{\@ifundefined{tenmib}{\global\font\tenmib\@mbi
   \global\font\tensyb\@mbsy
   \global\font\tenlyb\@lasyb\relax\@addfontinfo\@xpt
   {\def\boldmath{\everymath{\mit}\everydisplay{\mit}\@prtct\@nomathbold
        \textfont\@ne\tenmib \textfont\tw@\tensyb
        \textfont\lyfam\tenlyb \@prtct\@boldtrue}}}{}\@xpt\boldmath}%
\def\prm{\fam\z@\tenrm}%
\def\pit{\fam\itfam\tenit}\textfont\itfam\tenit \scriptfont\itfam\sevit
    \scriptscriptfont\itfam\sevit
\def\psl{\fam\slfam\tensl}\textfont\slfam\tensl
     \scriptfont\slfam\tensl \scriptscriptfont\slfam\tensl
\def\pbf{\fam\bffam\tenbf}\textfont\bffam\tenbf
    \scriptfont\bffam\sevenbf \scriptscriptfont\bffam\fivebf
\def\ptt{\fam\ttfam\tentt}\textfont\ttfam\tentt
    \scriptfont\ttfam\tentt \scriptscriptfont\ttfam\tentt
\def\psf{\fam\sffam\tensf}\textfont\sffam\tensf
    \scriptfont\sffam\tensf \scriptscriptfont\sffam\tensf
\def\psc{\@getfont\psc\scfam\@xpt{\@mcsc}}%
\def\ly{\fam\lyfam\tenly}\textfont\lyfam\tenly
   \scriptfont\lyfam\sevly \scriptscriptfont\lyfam\fivly
\@setstrut \rm}

\def\eqalign#1{\null\,\vcenter{\openup\jot\m@th
  \ialign{\strut\hfil$\displaystyle{##}$&$\displaystyle{{}##}$\hfil
      \crcr#1\crcr}}\,}
\def\eqaligncond#1{\null\,\vcenter{\openup\jot\m@th
  \ialign{\strut\hfil\rm##\quad&\hfil$\displaystyle{##}$&$\displaystyle
      {{}##}$\hfil&\quad##\hfil\crcr#1\crcr}}\,}
\def\doubleeqalign#1{\null\,\vcenter{\openup\jot\m@th
  \ialign{\strut\hfil$\displaystyle{##}$&$\displaystyle{{}##}$\hfil
         &\qquad\hfil$\displaystyle{##}$&$\displaystyle{{}##}$\hfil
      \crcr#1\crcr}}\,}

\font\tenrm=cmr10

\makeatother 

\font\teneuf=eufm10 \font\seveneuf=eufm7 \font\fiveeuf=eufm5
\font\nineeuf=eufm9 \font\sixeuf=eufm6
\newfam\frakfam
\def\tenfraktur{\textfont\frakfam=\teneuf \scriptfont\frakfam=\seveneuf
  \scriptscriptfont\frakfam=\fiveeuf}

\def\frak{\fam\frakfam}
\let\smallfrak=\ninefraktur
\tenfraktur

\catcode`|=\active
\def|{\ifmmode \vert \else \bgroup\small\def|{\null\egroup}\fi}

\def\R{{\bf R}}                     
\def\T{{\scriptscriptstyle\rm T}}   
\let\(=\langle
\let\)=\rangle
\def\vf{{\frak X}}                  
\def\geodesic#1{\gamma_{\lower1pt\hbox{$\scriptstyle#1$}}}
\def\phat{{\hat p}}
\def\xhat{{\hat x}}
\def\Thetahat{{\hat\Theta}}
\def\f{{\!f}}
\def\covD#1{\nabla_{\!#1}}          
\def\D#1{{\nabla\!#1}}              
\def\Dsqr#1{{\nabla^2\!#1}}         
\def\Xtilde{{\tilde X}}
\def\Ytilde{{\tilde Y}}
\def\Htilde{{\tilde H}}
\def\grad{\mathop{\rm grad}\nolimits}
\def\tr{\mathop{\rm tr}\nolimits}
\def\diag{\mathop{\rm diag}\nolimits}

\def\Riemann(#1,#2){\mathop{R(#1,#2)}}
\def\id{\mathop{\rm id}\nolimits}
\def\Ad{\mathop{\rm Ad}\nolimits}
\def\ad{\mathop{\rm ad}\nolimits}
\def\SO{\mathop{\it SO}\nolimits}
\def\so{\mathop{\frak so}\nolimits}
\def\gl{\mathop{\frak gl}\nolimits}
\mathchardef\rdot="0201             
\def\Ball{{B_\epsilon(\phat)}}
\def\Diff{{\mit\Delta}}
\def\vrem{{\mit\Xi}}
\def\hot{{\rm h.o.t.}}
\def\nupushf{{\nu_{\mskip-1.5mu*}\mskip-2mu f}}
\def\lyapunov{\tr\Theta^\T Q\Theta N}
\def\adjarrow{\mathop{\vcenter{\hbox{\dimen1=40pt \dimen3=1.5pt
  \raise\dimen3\rlap{\hbox to\dimen1{\rightarrowfill}}\lower\dimen3
  \hbox to\dimen1{\leftarrowfill}}}}}
\def\smallchoice#1{\mathchoice{{\textstyle#1}}{{\scriptstyle#1}}%
  {{\scriptscriptstyle#1}}{{\scriptscriptstyle#1}}}
\def\oneoverx#1{{1\over#1}}
\def\half{{\smallchoice{\oneoverx2}}}
\def\third{{\smallchoice{\oneoverx3}}}
\def\quarter{{\smallchoice{\oneoverx4}}}

\begin{document}

\title{Optimization Techniques on\\Riemannian Manifolds}
\author{Steven T. Smith}
\affiliation{Harvard University\\Division of Applied
Sciences\\Cambridge, Massachusetts 02138}
\maketitle

\begin{abstract}
The techniques and analysis presented in this paper provide new
methods to solve optimization problems posed on Riemannian manifolds.
A new point of view is offered for the solution of constrained
optimization problems.  Some classical optimization techniques on
Euclidean space are generalized to Riemannian manifolds.  Several
algorithms are presented and their convergence properties are analyzed
employing the Riemannian structure of the manifold.  Specifically, two
apparently new algorithms, which can be thought of as Newton's method
and the conjugate gradient method on Riemannian manifolds, are
presented and shown to possess, respectively, quadratic and
superlinear convergence.  Examples of each method on certain
Riemannian manifolds are given with the results of numerical
experiments.  Rayleigh's quotient defined on the sphere is one
example.  It is shown that Newton's method applied to this function
converges cubically, and that the Rayleigh quotient iteration is an
efficient approximation of Newton's method.  The Riemannian version of
the conjugate gradient method applied to this function gives a new
algorithm for finding the eigenvectors corresponding to the extreme
eigenvalues of a symmetric matrix.  Another example arises from
extremizing the function $\lyapunov$ on the special orthogonal group.
In a similar example, it is shown that Newton's method applied to the
sum of the squares of the off-diagonal entries of a symmetric matrix
converges cubically. \parfillskip=0pt
\end{abstract}


\section{Introduction}

The preponderance of optimization techniques address problems posed on
Euclidean spaces.  Indeed, several fundamental algorithms have arisen
from the desire to compute the minimum of quadratic forms on Euclidean
space.  However, many optimization problems are posed on non-Euclidean
spaces.  For example, finding the largest eigenvalue of a symmetric
matrix may be posed as the maximization of Rayleigh's quotient defined
on the sphere.  Optimization problems subject to nonlinear
differentiable equality constraints on Euclidean space also lie within
this category.  Many optimization problems share with these examples
the structure of a differentiable manifold endowed with a Riemannian
metric.  This is the subject of this paper: the extremization of
functions defined on Riemannian manifolds.

The minimization of functions on a Riemannian manifold is, at least
locally, equivalent to the smoothly constrained optimization problem
on a Euclidean space, because every $C^\infty$ Riemannian manifold can
be isometrically imbedded in some Euclidean
space~\cite[Vol.~V]{Spivak}.  However, the dimension of the Euclidean
space may be larger than the dimension of the manifold; practical and
aesthetic considerations suggest that one try to exploit the intrinsic
structure of the manifold.  Elements of this spirit may be found
throughout the field of numerical methods, such as the emphasis on
unitary (norm preserving) transformations in numerical linear
algebra~\cite{GVL}, or the use of feasible direction
methods~\cite{Fletcher,GillMurray,Sargent}.

An intrinsic approach leads one from the extrinsic idea of vector
addition to the exponential map and parallel translation, from
minimization along lines to minimization along geodesics, and from
partial differentiation to covariant differentiation.  The computation
of geodesics, parallel translation, and covariant derivatives can be
quite expensive. For an $n$-dimensional manifold, the computation of
geodesics and parallel translation requires the solution of a system
of $2n$ nonlinear and $n$ linear ordinary differential equations.
Nevertheless, many optimization problems are posed on manifolds that
have an underlying algebraic structure that may be exploited to
greatly reduce the complexity of these computations.  For example, on
a real compact semisimple Lie group endowed with its natural
Riemannian metric, geodesics and parallel translation may be computed
via matrix exponentiation~\cite{Helgason}. Several algorithms are
available to perform this computation~\cite{GVL,nineteendubious}.
This algebraic structure may be found in the problems posed by
Brockett~\cite{Brockett:match,Brockett:sort,Brockett:grad}, Bloch
et~al.~\cite{BBR1,BBR2}, Smith~\cite{Me}, Faybusovich~\cite{Leonid},
Lagarias~\cite{Lagarias}, Chu et~al.~\cite{Chu:sphere,Chu:grad},
Perkins et~al.~\cite{balreal}, and Helmke~\cite{Uwe}.  This approach
is also applicable if the manifold can be identified with a symmetric
space or, excepting parallel translation, a reductive homogeneous
space~\cite{KobayshiandNomizu,Nomizu}.  Perhaps the simplest
nontrivial example is the sphere, where geodesics and parallel
translation can be computed at low cost with trigonometric functions
and vector addition.  Furthermore, Brown and
Bartholomew-Biggs~\cite{Bcubed} show that in some cases function
minimization by following the solution of a system of ordinary
differential equations can be implemented such that it is competitive
with conventional techniques.

The outline of the paper is as follows. In Section~\ref{sec:prelim},
the optimization problem is posed and conventions to be held
throughout the paper are established.  The method of steepest descent
on a Riemannian manifold is described in Section~\ref{sec:steepdesc}.
To fix ideas, a proof of linear convergence is given.  The examples of
Rayleigh's quotient on the sphere and the function $\lyapunov$ on the
special orthogonal group are presented.  In Section~\ref{sec:newton},
Newton's method on a Riemannian manifold is derived.  As in Euclidean
space, this algorithm may be used to compute the extrema of
differentiable functions.  It is proved that this method converges
quadratically.  The example of Rayleigh's quotient is continued, and
it is shown that Newton's method applied to this function converges
cubically, and is approximated by the Rayleigh quotient iteration.
The example considering $\lyapunov$ is continued.  In a related
example, it is shown that Newton's method applied to the sum of the
squares of the off-diagonal elements of a symmetric matrix converges
cubically.  This provides an example of a cubically convergent
Jacobi-like method.  The conjugate gradient method is presented in
Section~\ref{sec:conjgrad} with a proof of superlinear convergence.
This technique is shown to provide an effective algorithm for
computing the extreme eigenvalues of a symmetric matrix.  The
conjugate gradient method is applied to the function $\lyapunov$.

\section{Preliminaries}
\label{sec:prelim}

This paper is concerned with the following problem.

\begin{problem}\label{mainproblem}\ignorespaces Let $M$ be a complete
Riemannian manifold, and $f$ a $C^\infty$ function on~$M$.  Compute
$$\min_{p\in M}f(p).$$
\end{problem}

There are many well-known algorithms for solving this problem in the
case where $M$ is a Euclidean space.  This paper generalizes several
of these algorithms to the case of complete Riemannian manifolds by
replacing the Euclidean notions of straight lines and ordinary
differentiation with geodesics and covariant differentiation.  These
concepts are reviewed in the following paragraphs.  We follow
Helgason's~\cite{Helgason} and Spivak's~\cite{Spivak} treatments of
covariant differentiation, the exponential map, and parallel
translation.  Details may be found in these references.

Let $M$ be a complete $n$-dimensional Riemannian manifold with
Riemannian structure $g$ and corresponding Levi-Civita connection
$\nabla$.  Denote the tangent plane at~$p$ in~$M$ by~$T_p$ or $T_pM$.
For every $p$ in~$M$, the Riemannian structure $g$ provides an inner
product on~$T_p$ given by the nondegenerate symmetric bilinear form
$g_p\colon T_p\times T_p\to\R$.  The notation $\(X,Y\)=g_p(X,Y)$ and
$\|X\|=g_p(X,X)^{1/2}$, where $X$, $Y\in T_p$, is often used.  The
distance between two points $p$ and $q$ in~$M$ is denoted by $d(p,q)$.
The gradient of a real-valued $C^\infty$ function $f$ on~$M$ at $p$,
denoted by $(\grad\f)_p$, is the unique vector in~$T_p$ such that
$df_p(X)=\((\grad\f)_p,X\)$ for~all $X$ in~$T_p$.

Denote the set of $C^\infty$ functions on~$M$ by~$C^\infty(M)$ and the
set of $C^\infty$ vector fields on~$M$ by~$\vf(M)$.  An affine
connection on~$M$ is a function~$\nabla$ which assigns to each vector
field $X\in\vf(M)$ an $\R$-linear map $\covD{X}\colon\vf(M)\to\vf(M)$
which satisfies $${\rm(i)}\quad\covD{f\!X+gY}=f\covD{X}+g\covD{Y},\qquad
{\rm(ii)}\quad\covD{X}(fY) =f\covD{X}Y+(X\f)Y,$$ for~all $f\!$, $g\in
C^\infty(M)$, $X$, $Y\in\vf(M)$.  The map $\covD{X}$ may be applied to
tensors of arbitrary type.  Let $\nabla$ be an affine connection
on~$M$ and $X\in\vf(M)$.  Then there exists a unique $\R$-linear map
$A\mapsto\covD{X}A$ of $C^\infty$ tensor fields into $C^\infty$ tensor
fields which satisfies
$$\vbox{\halign{\rm\hfil#\quad&#\hfil\qquad\qquad&\rm\hfil#\quad&#\hfil\cr
(i)&$\covD{X}f=X\f\!$,&(iv)&$\covD{X}$ preserves the type of tensors,\cr
(ii)&$\covD{X}Y$ is given by~$\nabla$,&(v)&$\covD{X}$ commutes with
contractions,\cr (iii)&\multispan3$\covD{X}$ is a derivation:
$\covD{X}(A\otimes B) =\covD{X}A\otimes B
+A\otimes\covD{X}B$,\hfil\cr}}$$ where $f\in C^\infty(M)$,
$Y\in\vf(M)$, and $A$, $B$ are $C^\infty$ tensor fields.  If $A$ is of
type~$(k,l)$, then $\covD{X}A$, called the covariant derivative of~$A$
along~$X$, is of type~$(k,l)$, and $\D A\colon X\mapsto\covD{X}A$,
called the covariant differential of~$A$, is of type~$(k,l+1)$.

Let $M$ be a differentiable manifold with affine connection~$\nabla$.
Let $\gamma\colon I\to M$ be a smooth curve with tangent vectors
$X(t)=\dot\gamma(t)$, where $I\subset\R$ is an open interval.  The
curve $\gamma$ is called a geodesic if $\covD{X}X=0$ for~all $t\in I$.
Let $Y(t)\in T_{\gamma(t)}$ ($t\in I$) be a smooth family of tangent
vectors defined along $\gamma$.  The family $Y(t)$ is said to be
parallel along $\gamma$ if $\covD{X}Y=0$ for~all $t\in I$.

For every $p$ in~$M$ and $X\ne0$ in~$T_p$, there exists a unique
geodesic $t\mapsto\geodesic{X}(t)$ such that $\geodesic{X}(0)=p$ and
$\dot\geodesic{X}(0)=X$.  We define the exponential map $\exp_p\colon
T_p\to M$ by $\exp_p(X)=\geodesic{X}(1)$ for~all $X\in T_p$ such that
$1$ is in the domain of~$\geodesic{X}$.  Oftentimes the map $\exp_p$
will be denoted by ``$\exp$'' when the choice of tangent plane is
clear, and $\geodesic{X}(t)$ will be denoted by $\exp tX$.  A
neighborhood $N_p$ of~$p$ in~$M$ is a normal neighborhood if $N_p=\exp
N_0$, where $N_0$ is a star-shaped neighborhood of the origin in~$T_p$
and $\exp$ maps $N_0$ diffeomorphically onto~$N_p$.  Normal
neighborhoods always exist.

Given a curve $\gamma\colon I\to M$ such that $\gamma(0)=p$, for each
$Y\in T_p$ there exists a unique family $Y(t)\in T_{\gamma(t)}$ ($t\in
I$) of tangent vectors parallel along $\gamma$ such that $Y(0)=Y$.  If
$\gamma$ joins the points $p$ and~$\gamma(\alpha)=q$, the parallelism
along $\gamma$ induces an isomorphism $\tau_{pq}\colon T_p\to T_q$
defined by $\tau_{pq}Y=Y(\alpha)$.

Let $M$ be a manifold with an affine connection $\nabla$, and $N_p$ a
normal neighborhood of~$p\in M$.  Define the vector field $\Xtilde$
on~$N_p$ adapted to the tangent vector $X$ in~$T_p$ by putting
$\Xtilde_q=\tau_{pq}X$, the parallel translation of $X$ along the
unique geodesic segment joining $p$ and $q$.

Given a Riemannian structure $g$ on~$M$, there exists a unique affine
connection $\nabla$ on~$M$, called the Levi-Civita connection, which
for~all $X$, $Y\in\vf(M)$ satisfies
$$\vbox{\halign{\rm\hfil#\quad&$\displaystyle#$\hfil\qquad&#\hfil\cr
(i)&\covD{X}Y-\covD{Y}X=[X,Y]&($\nabla$ is symmetric or
torsion-free),\cr (ii)&\nabla g=0&(parallel translation is an
isometry).\cr}}$$ Length minimizing curves on $M$ are geodesics of the
Levi-Civita connection.  We shall use this connection throughout the
paper.

Unless otherwise specified, all manifolds, vector fields, and
functions are assumed to be smooth.  When considering a function $f$
to be minimized, the assumption that $f$ is differentiable of
class~$C^\infty$ can be relaxed throughout the paper, but $f$ must be
continuously differentiable at least beyond the derivatives that
appear.  As the results of this paper are local ones, the assumption
that $M$ be complete may also be relaxed in certain instances.

We will use the the following definitions to compare the convergence
rates of various algorithms.

\begin{definition} Let $\{p_i\}$ be a Cauchy sequence in~$M$ that
converges to $\phat$.  (i)~The sequence $\{p_i\}$ is said to converge
(at least) {\it linearly\/} if there exists an integer $N$ and a
constant $\theta\in[0,1)$ such that $d(p_{i+1},\phat)\le\theta
d(p_i,\phat)$ for~all $i\ge N$.  (ii)~The sequence $\{p_i\}$ is said
to converge (at least) {\it quadratically\/} if there exists an
integer $N$ and a constant $\theta\ge0$ such that
$d(p_{i+1},\phat)\le\theta d^2(p_i,\phat)$ for~all $i\ge N$.
(iii)~The sequence $\{p_i\}$ is said to converge (at least) {\it
cubically\/} if there exists an integer $N$ and a constant
$\theta\ge0$ such that $d(p_{i+1},\phat)\le \theta d^3(p_i,\phat)$
for~all $i\ge{N}$. (iv)~The sequence $\{p_i\}$ is said to converge
{\it superlinearly\/} if it converges faster than any sequence that
converges linearly.
\end{definition}

\section{Steepest descent on Riemannian manifolds}
\label{sec:steepdesc}

The method of steepest descent on a Riemannian manifold is
conceptually identical to the method of steepest descent on Euclidean
space.  Each iteration involves a gradient computation and
minimization along the geodesic determined by the gradient.
Fletcher~\cite{Fletcher},
Botsaris~\cite{Botsaris:grad,Botsaris:class,Botsaris:geod}, and
Luenberger~\cite{Luenberger} describe this algorithm in Euclidean
space.  Gill and Murray~\cite{GillMurray} and Sargent~\cite{Sargent}
apply this technique in the presence of constraints.  In this section
we restate the method of steepest descent described in the literature
and provide an alternative formalism that will be useful in the
development of Newton's method and the conjugate gradient method on
Riemannian manifolds.

\begin{algorithm}[The method of steepest
descent]\label{al:steepman}\ignorespaces Let $M$ be a complete
Riemannian manifold with Riemannian structure $g$ and Levi-Civita
connection $\nabla$, and let $f\in C^\infty(M)$.
\begin{steps}
\step[0.] Select $p_0\in M$, compute
     $G_0=-(\grad\f)_{p_0}$, and set $i=0$.
\step[1.] Compute $\lambda_i$ such that
     $$f(\exp_{p_i}\lambda_iG_i)\le f(\exp_{p_i}\lambda G_i)$$ for all
     $\lambda\ge0$.
\step[2.] Set $$\eqalign{p_{i+1} &=\exp_{p_i}\lambda_iG_i,\cr
     G_{i+1} &=-(\grad\f)_{p_{i+1}},\cr}$$ increment $i$, and go~to Step~1.
\end{steps}
\end{algorithm}

It is easy to verify that $\(G_{i+1},\tau G_i\)=0$, for $i\ge0$, where
$\tau$ is the parallelism with respect to the geodesic from $p_i$ to
$p_{i+1}$.  By assumption, the function $\lambda\mapsto f(\exp\lambda
G_i)$ is minimized at $\lambda_i$. Therefore, we have $0
={(d/dt)|_{t=0}}\penalty0{f(\exp(\lambda_i+t)G_i)} =df_{p_{i+1}}(\tau
G_i) =\((\grad\f)_{p_{i+1}},\tau G_i\)$.  Thus the method of steepest
descent on a Riemannian manifold has the same deficiency as its
counterpart on a Euclidean space, i.e., it makes a ninety degree turn
at every step.

The convergence of Algorithm~\ref{al:steepman} is linear.  To prove
this fact, we will make use of a standard theorem of the calculus,
expressed in differential geometric language.  The covariant
derivative $\covD X\f$ of~$f$ along $X$ is defined to be $X\f$.  For
$k=1$, $2$,~\dots, define $\covD{X}^k\f
=\covD{X}\circ\cdots\circ\covD{X}\f$ ($k$~times), and let
$\covD{X}^0\f=f$.

\begin{remark}[Taylor's formula]\label{rem:taylorf}\ignorespaces Let
$M$ be a manifold with an affine connection $\nabla$, $N_p$ a normal
neighborhood of~$p\in M$, the vector field $\Xtilde$ on~$N_p$ adapted
to~$X$ in~$T_p$, and $f$ a $C^\infty$ function on~$M$.  Then there
exists an $\epsilon>0$ such that for every $\lambda\in[0,\epsilon)$
\begin{equation}\label{eq:taylorf}
\eqalign{f(\exp_p \lambda X) &=f(p) +\lambda(\covD\Xtilde\f)(p)
+\cdots+{\lambda^{n-1}\over(n-1)!}(\covD\Xtilde^{n-1}\f)(p)\cr
&\quad{}+{\lambda^n\over(n-1)!}\int_0^1(1-t)^{n-1}
(\covD\Xtilde^n\f)(\exp_p t\lambda X)\,dt.\cr}
\end{equation}
\end{remark}

\begin{proof} Let $N_0$ be a star-shaped neighborhood of $0\in
T_p$ such that $N_p=\exp N_0$.  There exists $\epsilon>0$ such that
$\lambda X\in N_0$ for~all $\lambda\in[0,\epsilon)$.  The map
$\lambda\mapsto f(\exp\lambda X)$ is a real $C^\infty$ function on
$[0,\epsilon)$ with derivative $(\covD\Xtilde\f)(\exp\lambda X)$.  The
statement follows by repeated integration by parts.
\end{proof}


The following special cases of Remark~\ref{rem:taylorf} will be
particularly useful.  When $n=2$, Eq.~(\ref{eq:taylorf}) yields
\begin{equation} \label{eq:taylorf2}
\eqalign{f(\exp_p\lambda X) &=f(p) +\lambda(\covD\Xtilde\f)(p)
+\lambda^2\int_0^1(1-t) (\covD\Xtilde^2\f)(\exp_pt\lambda X)\,dt.\cr}
\end{equation} Furthermore, when $n=1$, Eq.~(\ref{eq:taylorf})
applied to the function $\Xtilde\f=\covD\Xtilde\f$ yields
\begin{equation} \label{eq:dtaylorf1} (\Xtilde\f)(\exp_p\lambda X)
=(\Xtilde\f)(p) +\lambda\int_0^1 (\covD\Xtilde^2\f)(\exp_p t\lambda
X)\,dt. \end{equation}

The convergence proofs require a characterization of the second order
terms of~$f$ near a critical point.  Consider the second covariant
differential $\nabla\nabla\f=\nabla^2\f$ of a smooth function $f\colon
M\to\R$.  If $(U,x^1,\ldots,x^n)$ is a coordinate chart on~$M$, then
at~$p\in U$ this $(0,2)$ tensor takes the form
\begin{equation}\label{eq:d2f} (\nabla^2\f)_p =\sum_{i,j}
\biggl(\Bigl({\partial^2\f\over\partial x^i\partial x^j}\Bigr)_p
-\sum_k \Gamma_{ji}^k \Bigl({\partial\f\over\partial
x^k}\Bigr)_p\biggr)\, dx^i\otimes dx^j \end{equation} where
$\Gamma_{ij}^k$ are the Christoffel symbols at~$p$.  If $\phat$ in~$U$
is a critical point of~$f\!$, then $(\partial\f/\partial x^k)_\phat=0$,
$k=1$,~\dots, $n$. Therefore $(\nabla^2\f)_\phat=(d^2\f)_\phat,$ where
$(d^2\f)_\phat$ is the Hessian of~$f$ at the critical point $\phat$.
Furthermore, for $p\in M$, $X$, $Y\in T_p$, and $\Xtilde$ and
$\Ytilde$ vector fields adapted to $X$ and $Y$, respectively, on a
normal neighborhood $N_p$ of~$p$, we have $(\nabla^2\f)
(\Xtilde,\Ytilde) =\covD\Ytilde\covD\Xtilde\f$ on~$N_p$.  Therefore
the coefficient of the second term of the Taylor expansion of $f(\exp
tX)$ is $(\covD\Xtilde^2\f)_p =(\nabla^2\f)_p(X,X)$.  Note that the
bilinear form $(\nabla^2\f)_p$ on $T_p\times T_p$ is symmetric if and
only if $\nabla$ is symmetric, which true of the Levi-Civita
connection by definition.

\begin{theorem}\label{th:linearconv}\ignorespaces Let $M$ be a complete
Riemannian manifold with Riemannian structure $g$ and Levi-Civita
connection $\nabla$.  Let $f\in C^\infty(M)$ have a nondegenerate
critical point at $\phat$ such that the Hessian $(d^2\f)_\phat$ is
positive definite.  Let $p_i$ be a sequence of points in $M$
converging to~$\phat$ and $H_i\in T_{p_i}$ a sequence of tangent
vectors such that $$\eqaligncond{(i)& p_{i+1}
&=\smash{\exp_{p_i}}\lambda_iH_i &for $i=0$, $1$,~\dots,\cr (ii)&
\(-(\grad\f)_{p_i},H_i\) &\ge c\,\|(\grad\f)_{p_i}\|\>\|H_i\| &for
$c\in(0,1]$,\cr}$$ where $\lambda_i$ is chosen such that
$f(\exp\lambda_iH_i)\le f(\exp\lambda H_i)$ for all $\lambda\ge0$.
Then there exists a constant $E$ and a $\theta\in[0,1)$ such that
for~all $i=0$, $1$,~\dots, $$d(p_i,\phat)\le E\theta^i.$$
\end{theorem}

\begin{proof} The proof is a generalization of the one given in
Polak~\cite[p.~242ff\/]{Polak} for the method of steepest descent on
Euclidean space.

The existence of a convergent sequence is guaranteed by the smoothness
of~$f$.  If $p_j=\phat$ for some integer $j$, the assertion becomes
trivial; assume otherwise.  By the smoothness of~$f\!$, there exists an
open neighborhood $U$ of~$\phat$ such that $(\nabla^2\f)_p$ is
positive definite for~all $p\in U$. Therefore, there exist constants
$k>0$ and $K\ge k>0$ such that for all $X\in T_p$ and all $p\in U$,
\begin{equation} \label{eq:d2fbounds} k\|X\|^2\le
(\nabla^2\f)_p(X,X)\le K\|X\|^2. \end{equation}

Define $X_i\in T_{\phat}$ by the relations $\exp X_i=p_i$, $i=0$, $1$,
\dots\spacefactor=3000\relax\space By assumption, $df_\phat=0$ and
from Eq.~(\ref{eq:taylorf2}), we have \begin{equation}
\label{eq:taylorfsp} f(p_i)-f(\phat) =\int_0^1
(1-t)(\covD{\Xtilde_i}^2f)(\exp_\phat tX_i)\,dt.\end{equation}
Combining this equality with the inequalities of (\ref{eq:d2fbounds})
yields \begin{equation} \label{eq:fbounds} \half kd^2(p_i,\phat)\le
f(p_i)-f(\phat)\le \half Kd^2(p_i,\phat). \end{equation} Similarly, we
have by Eq.~(\ref{eq:dtaylorf1}) $$(\Xtilde_if)(p_i) =\int_0^1
(\covD{\Xtilde_i}^2f)(\exp_\phat tX_i)\,dt.$$ Next, use
(\ref{eq:taylorfsp}) with Schwarz's inequality and the first
inequality of~(\ref{eq:fbounds}) to obtain
$$\eqalign{kd^2(p_i,\phat)=k\|X_i\|^2 &\le\int_0^1
(\covD{\Xtilde_i}^2f)(\exp_\phat tX_i)\,dt =(\Xtilde_if)(p_i)\cr
&=df_{p_i}\bigl((\Xtilde_i)_{p_i}\bigr) =df_{p_i}(\tau X_i)
=\((\grad\f)_{p_i},\tau X_i\)\cr &\le
\|(\grad\f)_{p_i}\|\>\|\tau X_i\| =\|(\grad\f)_{p_i}\|\>
d(p_i,\phat).\cr}$$ Therefore, \begin{equation} \label{eq:gradlb}
\|(\grad\f)_{p_i}\|\ge kd(p_i,\phat).\end{equation}

Define the function $\Diff\colon T_p\times\R\to\R$ by the equation
$\Diff(X,\lambda)=f(\exp_p\lambda X)-f(p)$. By
Eq.~(\ref{eq:taylorf2}), the second order Taylor formula, we have
$$\Diff(H_i,\lambda) =\lambda(\Htilde_if)(p_i) +\half\lambda^2\int_0^1
(1-t)(\covD{\Htilde_i}^2f)(\exp_{p_i}\lambda H_i)\,dt.$$ Using
assumption~(ii) of the theorem along with (\ref{eq:d2fbounds}) we
establish for $\lambda\ge0$ \begin{equation}\label{eq:diffbounds}
\Diff(H_i,\lambda)\le -\lambda c\|(\grad\f)_{p_i}\|\>\|H_i\|
+\half\lambda^2K\|H_i\|^2. \end{equation}

We may now compute an upper bound for the rate of linear convergence
$\theta$.  By assumption~(i) of the theorem, $\lambda$ must be chosen
to minimize the right hand side of~(\ref{eq:diffbounds}). This
corresponds to choosing $\lambda=c\|(\grad\f)_{p_i}\|\big/K\|H_i\|$.
A computation reveals that $$\Diff(H_i,\lambda_i)\le
-{c^2\over2K}\|(\grad\f)_{p_i}\|^2.$$ Applying (\ref{eq:fbounds}) and
(\ref{eq:gradlb}) to this inequality and rearranging terms yields
\begin{equation} \label{eq:fconv} f(p_{i+1})-f(\phat)\le
\theta\bigl(f(p_i)-f(\phat)\bigr), \end{equation} where
$\theta=\bigl(1-(ck/K)^2\bigr)$.  By assumption, $c\in(0,1]$ and
$0<k\le K$, therefore $\theta\in[0,1)$.  (Note that Schwarz's
inequality bounds $c$ below unity.)  From (\ref{eq:fconv}) it is seen
that $\bigl(f(p_i)-f(\phat)\bigr)\le E\theta^i$ where
$E=\bigl(f(p_0)-f(\phat)\bigr)$.  From (\ref{eq:fbounds}) we conclude
that for $i=0$, $1$,~\dots, \begin{equation} \label{eq:linearconv}
d(p_i,\phat)\le \sqrt{2E\over k}\bigl(\sqrt\theta\,\bigr)^i.\tombstone
\end{equation}
\end{proof}

\begin{corollary} If Algorithm~\ref{al:steepman} converges to a local
minimum, it converges linearly.
\end{corollary}

The choice $H_i=-(\grad\f)_{p_i}$ yields $c=1$ in the second
assumption the Theorem~\ref{th:linearconv}, which establishes the
corollary.

\begin{example}[Rayleigh's quotient on the
sphere]\label{eg:rayqgrad}\ignorespaces Let $S^{n-1}$ be the imbedded
sphere in~$\R^n$, i.e., $S^{n-1}=\{\,x\in\R^n:x^\T x=1\,\}$, where
$x^\T y$ denotes the standard inner product on~$\R^n$, which induces a
metric on~$S^{n-1}$.  Geodesics on the sphere are great circles and
parallel translation along geodesics is equivalent to rotating the
tangent plane along the great circle.  Let $x\in S^{n-1}$ and $h\in
T_x$ have unit length, and $v\in T_x$ be any tangent vector.  Then
$$\eqalign{\exp_x th &=x\cos t +h\sin t,\cr \tau h &=h\cos t -x\sin
t,\cr \tau v &=v-(h^\T v)\bigl(x\sin t +h(1-\cos t)\bigr),\cr}$$ where
$\tau$ is the parallelism along the geodesic $t\mapsto\exp th$.  Let
$Q$ be an \hbox{$n$-by-$n$} positive definite symmetric matrix with
distinct eigenvalues and define $\rho\colon S^{n-1}\to\R$ by
$\rho(x)=x^\T Qx$.  A computation shows that
\begin{equation}\label{eq:raygrad}
\half(\grad\rho)_x =Qx-\rho(x)x. \end{equation} The function $\rho$
has a unique minimum and maximum point at the eigenvectors
corresponding to the smallest and largest eigenvalues of~$Q$,
respectively. Because $S^{n-1}$ is geodesically complete, the method
of steepest descent in the opposite direction of the gradient
converges to the eigenvector corresponding to the smallest eigenvalue
of~$Q$; likewise for the eigenvector corresponding to the largest
eigenvalue.  Chu~\cite{Chu:sphere} considers the continuous limit of
this problem.  A computation shows that $\rho(x)$ is maximized along
the geodesic $\exp_x th$ ($\|h\|=1$) when $a\cos2t-b\sin2t=0$, where
$a=2x^\T Qh$ and $b=\rho(x)-\rho(h)$.  Thus $\cos t$ and $\sin t$ may
be computed with simple algebraic functions of $a$ and~$b$ (which
appear below in Algorithm~\ref{al:raysphere}).  The results of a
numerical experiment demonstrating the convergence of the method of
steepest descent applied to maximizing Rayleigh's quotient on~$S^{20}$
are shown in Figure~\ref{fig:raysphere} on page~\pageref{fig:raysphere}.
\end{example}

\def\temp{\cite{Brockett:sort,Brockett:grad}}
\begin{example}[Brockett~\temp]\label{eg:trHNgrad}\ignorespaces
Consider the function $f(\Theta)=\lyapunov$ on the special orthogonal
group $\SO(n)$, where $Q$ is a real symmetric matrix with distinct
eigenvalues and $N$ is a real diagonal matrix with distinct diagonal
elements. It will be convenient to identify tangent vectors in
$T_\Theta$ with tangent vectors in $T_I\cong\so(n)$, the tangent plane
at the identity, via left translation.  The gradient of~$f$ (with
respect to the negative Killing form of~$\so(n)$, scaled by~$1/(n-2)$)
at $\Theta\in\SO(n)$ is $\Theta[H,N]$, where
$H=\Ad_{\Theta^\T}(Q)=\Theta^\T Q\Theta$.  The group $\SO(n)$ acts on
the set of symmetric matrices by conjugation; the orbit of~$Q$ under
the action of~$\SO(n)$ is an isospectral submanifold of the symmetric
matrices.  We seek a $\Thetahat$ such that $f(\Thetahat)$ is
maximized.  This point corresponds to a diagonal matrix whose diagonal
entries are ordered similarly to those of~$N$.  A related example is
found in Smith~\cite{Me}, who considers the homogeneous space of
matrices with fixed singular values, and in Chu~\cite{Chu:grad}.

The Levi-Civita connection on~$\SO(n)$ is bi-invariant and invariant
with respect to inversion; therefore, geodesics and parallel
translation may be computed via matrix exponentiation of elements
in~$\so(n)$ and left (or right) translation~\cite[Ch.~II,
Ex.~6]{Helgason}.  The geodesic emanating from the identity in
$\SO(n)$ in direction $X\in\so(n)$ is given by the formula
$\exp_ItX=e^{tX}$, where the right hand side denotes regular matrix
exponentiation.  The expense of geodesic minimization may be avoided
if instead one uses Brockett's estimate~\cite{Brockett:grad} for the
step size.  Given $\Omega\in\so(n)$, we wish to find $t>0$ such that
$\phi(t)=\tr\Ad_{e^{-t\Omega}}(H)N$ is minimized.  Differentiating
$\phi$ twice shows that $\phi'(t)=-\tr\Ad_{e^{-t\Omega}}(\ad_\Omega
H)N$ and $\phi''(t)=-\tr\Ad_{e^{-t\Omega}}(\ad_\Omega H)\ad_\Omega N$,
where $\ad_\Omega A=[\Omega,A]$.  Hence, $\phi'(0)=2\tr H\Omega N$
and, by Schwarz's inequality and the fact that $\Ad$ is an isometry,
$|\phi''(t)|\le \|\ad_\Omega H\|\;\|\ad_\Omega N\|$.  We conclude that
if $\phi'(0)>0$, then $\phi'$ is nonnegative on the interval
\begin{equation}\label{eq:Brockettest} 0\le t\le {2\tr H\Omega
N\over\|\ad_\Omega H\|\;\|\ad_\Omega N\|}, \end{equation} which
provides an estimate for the step size of Step~1 in
Algorithm~\ref{al:steepman}.  The results of a numerical experiment
demonstrating the convergence of the method of steepest descent
(ascent) in~$\SO(20)$ using this estimate are shown in
Figure~\ref{fig:SOnconv}.
\end{example}

\section{Newton's method on Riemannian manifolds}
\label{sec:newton}

As in the optimization of functions on Euclidean space, quadratic
convergence can be obtained if the second order terms of the Taylor
expansion are used appropriately.  In this section we present Newton's
algorithm on Riemannian manifolds, prove that its convergence is
quadratic, and provide examples.  Whereas the convergence proof for
the method of steepest descent relies upon the Taylor expansion of the
function $f\!$, the convergence proof for Newton's method will rely upon
the Taylor expansion of the one-form $df$.  Note that Newton's method
has a counterpart in the theory of constrained optimization, as
described by, e.g., Fletcher~\cite{Fletcher},
Bertsekas~\cite{Bertsekas:newton,Bertsekas}, or
Dunn~\cite{Dunn:newton,Dunn:grad}.  The Newton method presented in
this section has only local convergence properties. There is a theory
of global Newton methods on Euclidean space and computational
complexity; see the work of Hirsch and Smale~\cite{HirschSmale},
Smale~\cite{Smale:fta,Smale:eff}, and Shub and
Smale~\cite{ShubSmale:I,ShubSmale:II}.

Let $M$ be an $n$-dimensional Riemannian manifold with Riemannian
structure $g$ and Levi-Civita connection $\nabla$, let $\mu$ be a
$C^\infty$ one-form on~$M$, and let $p$ in~$M$ be such that the
bilinear form $(\D\mu)_p\colon T_p\times T_p\to\R$ is nondegenerate.
Then, by abuse of notation, we have the pair of isomorphisms
$$T_p\mathrel{\adjarrow^{(\D\mu)_p}_
{(\D\mu)_p^{\rlap{$\scriptscriptstyle-1$}}}}T_p^*$$ with the forward
map defined by $X\mapsto (\covD X\mu)_p= (\D\mu)_p(\rdot,X)$, which is
nonsingular.  The notation $(\D\mu)_p$ will henceforth be used for
both the bilinear form defined by the covariant differential of~$\mu$
evaluated at~$p$ and the homomorphism from $T_p$ to~$T_p^*$ induced by
this bilinear form.  In case of an isomorphism, the inverse can be
used to compute a point in~$M$ where $\mu$ vanishes, if such a point
exists.  The case $\mu=df$ will be of particular interest, in which
case $\D\mu=\nabla^2\f$.  Before expounding on these ideas, we make
the following remarks.

\begin{remark}[The mean value theorem]\label{rem:mvt}\ignorespaces Let
$M$ be a manifold with affine connection~$\nabla$, $N_p$ a normal
neighborhood of~$p\in M$, the vector field $\Xtilde$ on~$N_p$ adapted
to~$X\in T_p$, $\mu$ a one-form on~$N_p$, and $\tau_\lambda$ the
parallelism with respect to $\exp tX$ for $t\in[0,\lambda]$.  Denote
the point $\exp\lambda X$ by~$p_\lambda$.  Then there exists an
$\epsilon>0$ such that for every $\lambda\in[0,\epsilon)$, there is an
$\alpha\in[0,\lambda]$ such that $$\tau_\lambda^{-1}\mu_{p_\lambda}
-\mu_p =\lambda(\covD\Xtilde\mu)_{p_\alpha}\circ\tau_\alpha.$$
\end{remark}

\begin{proof} As in the proof of Remark~\ref{rem:taylorf},
there exists an $\epsilon>0$ such that $\lambda X\in N_0$ for~all
$\lambda\in[0,\epsilon)$.  The map $\lambda\mapsto
(\tau_\lambda^{-1}\mu_{p_\lambda})(A)$, for any $A$ in~$T_p$, is a
$C^\infty$ function on~$[0,\epsilon)$ with derivative
$(d/dt)(\tau_t^{-1}\mu_{p_t})(A) =(d/dt)\mu_{p_t}(\tau_tA)
=\covD\Xtilde\bigl(\mu_{p_t}(\tau_tA)\bigr)
=(\covD\Xtilde\mu)_{p_t}(\tau_tA)+\mu_{p_t}\bigl(\covD\Xtilde(\tau_tA)\bigr)
=(\covD\Xtilde\mu)_{p_t}(\tau_tA)$.  The lemma follows from the mean
value theorem of real analysis.
\end{proof}

This remark can be generalized in the following way.

\begin{remark}[Taylor's theorem]\label{rem:taylort}\ignorespaces Let
$M$ be a manifold with affine connection~$\nabla$, $N_p$ a normal
neighborhood of~$p\in M$, the vector field $\Xtilde$ on~$N_p$ adapted
to~$X\in T_p$, $\mu$ a one-form on~$N_p$, and $\tau_\lambda$ the
parallelism with respect to $\exp tX$ for $t\in[0,\lambda]$.  Denote
the point $\exp\lambda X$ by~$p_\lambda$.  Then there exists an
$\epsilon>0$ such that for every $\lambda\in[0,\epsilon)$, there is an
$\alpha\in[0,\lambda]$ such that \begin{equation}\label{eq:taylort}
\tau_\lambda^{-1}\mu_{p_\lambda} =\mu_p +\lambda(\covD\Xtilde\mu)_p
+\cdots+{\lambda^{n-1}\over(n-1)!} (\covD\Xtilde^{n-1}\mu)_p
+{\lambda^n\over n!} (\covD\Xtilde^n\mu)_{p_\alpha}\circ\tau_\alpha.
\end{equation}
\end{remark}

The remark follows by applying Remark~\ref{rem:mvt} and the Taylor's
theorem of real analysis to the function $\lambda\mapsto
(\tau_\lambda^{-1}\mu_{p_\lambda})(A)$ for any $A$ in~$T_p$.

Remarks \ref{rem:mvt} and \ref{rem:taylort} can be generalized to
$C^\infty$ tensor fields, but we will only require
Remark~\ref{rem:taylort} for case $n=2$ to make the following
observation.

Let $\mu$ be a one-form on~$M$ such that for some $\phat$ in~$M$,
$\mu_\phat=0$.  Given any $p$ in a normal neighborhood of~$\phat$, we
wish to find $X$ in~$T_p$ such that $\exp_pX=\phat$.  Consider the
Taylor expansion of~$\mu$ about $p$, and let $\tau$ be the parallel
translation along the unique geodesic joining $p$ to~$\phat$. We have
by our assumption that $\mu$ vanishes at~$\phat$, and from
Eq.~(\ref{eq:taylort}) for $n=2$, $$0=\tau^{-1}\mu_\phat
=\tau^{-1}\mu_{\exp_pX} =\mu_p +(\D\mu)_p(\rdot,X) +\hot$$ If the
bilinear form $(\D\mu)_p$ is nondegenerate, the tangent vector $X$ may
be approximated by discarding the higher order terms and solving the
resulting linear equation $$\mu_p +(\D\mu)_p(\rdot,X) =0$$ for~$X$,
which yields $$X=-(\D\mu)_p^{-1}\mu_p.$$ This approximation is the
basis of the following algorithm.

\begin{algorithm}[Newton's method]\label{al:newtonman}\ignorespaces Let
$M$ be a complete Riemannian manifold with Riemannian structure $g$
and Levi-Civita connection $\nabla$, and let $\mu$ be a $C^\infty$
one-form on~$M$.
\begin{steps}
\step[0.] Select $p_0\in M$ such that $(\D\mu)_{p_0}$ is
     nondegenerate, and set $i=0$.
\step[1.] Compute $$\eqalign{H_i
     &=-(\D\mu)_{p_i}^{-1}\mu_{p_i}\cr p_{i+1} &=\exp_{p_i}H_i,\cr}$$
     (assume that $(\D\mu)_{p_i}$ is nondegenerate), increment $i$, and
     repeat.
\end{steps}
\end{algorithm}

It can be shown that if $p_0$ is chosen suitably close (within the
so-called domain of attraction) to a point $\phat$ in~$M$ such
that $\mu_\phat=0$ and $(\D\mu)_\phat$ is nondegenerate, then
Algorithm~\ref{al:newtonman} converges quadratically to~$\phat$.  The
following theorem holds for general one-forms; we will consider the
case where $\mu$ is exact.

\begin{theorem}\label{th:newtonquad}\ignorespaces Let $f\in
C^\infty(M)$ have a nondegenerate critical point at $\phat$.  Then
there exists a neighborhood $U$ of~$\phat$ such that for any $p_0\in
U$, the iterates of Algorithm~\ref{al:newtonman} for $\mu=df$ are well
defined and converge quadratically to~$\phat$.
\end{theorem}

The proof of this theorem is a generalization of the corresponding
proof for Euclidean spaces, with an extra term containing the
Riemannian curvature tensor (which of course vanishes in the latter
case).

\begin{proof} If $p_j=\phat$ for some integer $j$, the assertion
becomes trivial; assume otherwise.  Define $X_i\in T_{p_i}$ by the
relations $\phat=\exp X_i$, $i=0$, $1$, \dots, so that
$d(p_i,\phat)=\|X_i\|$ (n.b.\ this convention is opposite that used in
the proof of Theorem~\ref{th:linearconv}).  Consider the geodesic
triangle with vertices $p_i$, $p_{i+1}$, and $\phat$, and sides $\exp
tX_i$ from $p_i$ to~$\phat$, $\exp tH_i$ from $p_i$ to~$p_{i+1}$, and
$\exp tX_{i+1}$ from $p_{i+1}$ to~$\phat$, for $t\in[0,1]$.  Let
$\tau$ be the parallelism with respect to the side $\exp tH_i$ between
$p_i$ and $p_{i+1}$.  There exists a unique tangent vector $\vrem_i$
in~$T_{p_i}$ defined by the equation \begin{equation}\label{eq:vrem}
X_i=H_i+\tau^{-1}X_{i+1}+\vrem_i \end{equation} ($\vrem_i$ may be
interpreted as the amount by which vector addition fails).  If we use
the definition $H_i=-(\nabla^2\f)_{p_i}^{-1}df_{p_i}$ of
Algorithm~\ref{al:newtonman}, apply the isomorphism
$(\nabla^2\f)_{p_i}\colon T_{p_i}\to T_{p_i}^*$ to both sides of
Eq.~(\ref{eq:vrem}), we obtain the equation
\begin{equation}\label{eq:quadproof1} (\nabla^2\f)_{p_i}
(\tau^{-1}X_{i+1}) =df_{p_i} +(\nabla^2\f)_{p_i}X_i
-(\nabla^2\f)_{p_i}\vrem_i. \end{equation} By Taylor's theorem, there
exists an $\alpha\in[0,1]$ such that
\begin{equation}\label{eq:quadproof2} \tau_1^{-1}df_\phat =df_{p_i}
+(\covD{\Xtilde_i}df)_{p_i} +\half(\covD{\Xtilde_i}^2df)_{p_\alpha}
\circ\tau_\alpha \end{equation} where $\tau_t$ is the parallel
translation from~$p_i$ to~$p_t=\exp tX_i$.  The trivial identities
$(\covD{\Xtilde_i}df)_{p_i} =(\nabla^2\f)_{p_i}X_i$ and
$(\covD{\Xtilde_i}^2df)_{p_\alpha} =(\nabla^3\f)_{p_\alpha}
(\tau_\alpha\rdot,\tau_\alpha X_i,\tau_\alpha X_i)$ will be used to
replace the last two terms on the right hand side of
Eq.~(\ref{eq:quadproof2}).  Combining the assumption that $df_\phat=0$
with Eqs.\ (\ref{eq:quadproof1}) and (\ref{eq:quadproof2}), we obtain
\begin{equation}\label{eq:quadproof3} (\nabla^2\f)_{p_i}
(\tau^{-1}X_{i+1})= -\half(\covD{\Xtilde_i}^2df)_{p_\alpha}
\circ\tau_\alpha -(\nabla^2\f)_{p_i}\vrem_i. \end{equation}

By the smoothness of~$f$ and $g$, there exists an $\epsilon>0$ and
constants $\delta'$, $\delta''$, $\delta'''$, all greater than zero,
such that whenever $p$ is in the convex normal ball $\Ball$,
$$\eqaligncond{(i)& \|(\nabla^2\f)_p(\rdot,X)\|&\ge\delta'\|X\|
&for~all $X\in T_p$,\cr (ii)&
\|(\nabla^2\f)_p(\rdot,X)\|&\le\delta''\|X\| &for~all $X\in T_p$,\cr
(iii)& \|(\nabla^3\f)_p(\rdot,X,X)\|&\le\delta'''\|X\|^2 &for~all
$X\in T_p$,\cr}$$ where the induced norm on~$T_p^*$ is used in all
three cases.  Taking the norm of both sides of
Eq.~(\ref{eq:quadproof3}), applying the triangle inequality to the
right hand side, and using the fact that parallel translation is an
isometry, we obtain the inequality
\begin{equation}\label{eq:quadproof4} \delta'd(p_{i+1},\phat)\le
\delta'''d^2(p_i,\phat) +\delta''\|\vrem_i\|. \end{equation}

The length of~$\vrem_i$ can be bounded by a cubic expression
in~$d(p_i,\phat)$ by considering the distance between the points
$\exp(H_i+\tau^{-1}X_{i+1})$ and $\exp X_{i+1}=\phat$.  Given $p\in
M$, $\epsilon>0$ small enough, let $a$, $v\in T_p$ be such that
$\|a\|+\|v\|\le\epsilon$, and let $\tau$ be the parallel translation
with respect to the geodesic from~$p$ to~$q=\exp_p a$.
Karcher~\cite[App.~C2.2]{Karcher} shows that \begin{equation}
\label{eq:Karcher} d\bigl(\exp_p(a+v), \exp_q(\tau v)\bigr)\le
\|a\|\cdot{\rm const.}\,(\max|K|)\cdot\epsilon^2,\end{equation} where
$K$ is the sectional curvature of~$M$ along any section in the tangent
plane at any point near~$p$.

There exists a constant $c>0$ such that $\|\vrem_i\|\le c\,
d\bigl(\phat, {\exp(H_i +\tau^{-1}X_{i+1})}\bigr)$.
By~(\ref{eq:Karcher}), we have $\|\vrem_i\|\le {\rm
const.}\,\|H_i\|\epsilon^2$.  Taking the norm of both sides of the
Taylor formula $df_{p_i} =-\int_0^1 (\covD{\Xtilde_i}df) (\exp
tX_i)\,dt$ and applying a standard integral inequality and
inequality~(ii) from above yields $\|df_{p_i}\|\le\delta''\|X_i\|$ so
that $\|H_i\|\le {\rm const.}\,\|X_i\|$.  Furthermore, we have the
triangle inequality $\|X_{i+1}\|\le \|X_i\|+\|H_i\|$, therefore
$\epsilon$ may be chosen such that $\|H_i\|+\|X_{i+1}\|\le\epsilon\le
{\rm const.}\,\|X_i\|$.  By~(\ref{eq:Karcher}) there exists
$\delta^{\rm iv}>0$ such that $\|\vrem_i\|\le\delta^{\rm
iv}d^3(p_i,\phat)$.
\end{proof}

\begin{corollary} If\/ $(\nabla^2\f)_\phat$ is positive (negative)
definite and Algorithm~\ref{al:newtonman} converges to~$\phat$, then
Algorithm~\ref{al:newtonman} converges quadratically to a local
minimum (maximum) of~$f$.
\end{corollary}

\begin{example}[Rayleigh's quotient on the
sphere]\label{eg:newtonray}\ignorespaces Let $S^{n-1}$ and
$\rho(x)=x^\T Qx$ be as in Example~\ref{eg:rayqgrad}. It will be
convenient to work with the coordinates $x^1$,~\dots, $x^n$ of the
ambient space~$\R^n$, treat the tangent plane $T_xS^{n-1}$ as a vector
subspace of~$\R^n$, and make the identification $T_xS^{n-1}\cong
T_x^*S^{n-1}$ via the metric. In this coordinate system, geodesics on
the sphere obey the second order differential equation $\ddot
x^k+x^k=0$, $k=1$,~\dots, $n$. Thus the Christoffel symbols are given
by $\Gamma_{ij}^k=\delta_{ij}x^k$, where $\delta_{ij}$ is the
Kronecker delta.  The $ij$th component of the second covariant
differential of~$\rho$ at~$x$ in~$S^{n-1}$ is given by (cf.\
Eq.~(\ref{eq:d2f})) $$\bigl((\Dsqr\rho)_x\bigr)_{ij} =2Q_{ij}
-\sum_{k,l}\delta_{ij}x^k\cdot 2Q_{kl}x^l =2\bigl(Q_{ij}
-\rho(x)\delta_{ij}\bigr),$$ or, written as matrices,
\begin{equation}\label{eq:d2ray} \half(\Dsqr\rho)_x=Q-\rho(x)I.
\end{equation} Let $u$ be a tangent vector in~$T_x S^{n-1}$.
A linear operator $A\colon\R^n\to\R^n$ defines a linear operator on
the tangent plane $T_xS^{n-1}$ for each $x$ in~$S^{n-1}$ such that
$$A\rdot u=Au-(x^\T\!Au)x =(I-xx^\T)Au$$ If $A$ is invertible as an
endomorphism of the ambient space $\R^n$, the solution to the linear
equation $A\rdot u=v$ for $u$, $v$ in~$T_xS^{n-1}$ is
\begin{equation}\label{eq:tanginv} u=A^{-1}\left(v
-{(x^\T\!A^{-1}v)\over(x^\T\!A^{-1}x)}x\right).
\end{equation} For Newton's method, the direction
$H_i$ in~$T_xS^{n-1}$ is the solution of the equation
$$(\Dsqr\rho)_{x_i}\rdot H_i= -(\grad\rho)_{x_i}.$$ Combining Eqs.\
(\ref{eq:raygrad}), (\ref{eq:d2ray}), and (\ref{eq:tanginv}), we
obtain $$H_i =-x_i +\alpha_i\bigl(Q-\rho(x_i)I\bigr)^{-1}x_i$$ where
$\alpha_i=1\big/x_i^\T(Q-\rho(x_i)I)^{-1}x_i$.  This gives rise to the
following algorithm for computing eigenvectors of the symmetric
matrix~$Q$.
\end{example}

\begin{algorithm}[Newton-Rayleigh quotient
method]\label{al:newtonray}\ignorespaces Let $Q$ be a real symmetric
\hbox{$n$-by-$n$} matrix.
\begin{steps}
\step[0.] Select $x_0$ in~$\R^n$ such that $x_0^\T x_0=1$, and
     set $i=0$.
\step[1.] Compute $$y_i=\bigl(Q-\rho(x_i)I\bigr)^{-1}x_i$$ and
     set $\alpha_i=1\big/x_i^\T y_i$.
\step[2.] Compute $$\eqalign{H_i &=-x_i+\alpha_iy_i,\quad
     \theta_i=\|H_i\|,\cr
     x_{i+1} &=x_i\cos\theta_i +H_i\sin\theta_i/\theta_i,\cr}$$
     increment $i$, and go~to Step~1.
\end{steps}
\end{algorithm}

The quadratic convergence guaranteed by Theorem~\ref{th:newtonquad} is
in fact too conservative for Algorithm~\ref{al:newtonray}.  As
evidenced by Figure~\ref{fig:raysphere}, Algorithm~\ref{al:newtonray}
converges cubically.

\begin{proposition} If $\lambda$ is a distinct eigenvalue of the
symmetric matrix~$Q$, and Algorithm~\ref{al:newtonray} converges to
the corresponding eigenvector $\xhat$, then it converges cubically.
\end{proposition}

\begin{proof}[Proof 1] In the coordinates $x^1$,~\dots, $x^n$ of the
ambient space $\R^n$, the $ijk$th component of the third covariant
differential of~$\rho$ at~$\xhat$ is $-2\lambda \xhat^k\delta_{ij}$.
Let $X\in T_\xhat S^{n-1}$.  Then $(\nabla^3\rho)_\xhat(\rdot,X,X)=0$
and the second order terms on the right hand side of
Eq.~(\ref{eq:quadproof3}) vanish at the critical point.  The
proposition follows from the smoothness of~$\rho$.
\end{proof}

\begin{proof}[Proof 2] The proof follows
Parlett's~\cite[p.~72ff\/]{Parlett} proof of cubic convergence for the
Rayleigh quotient iteration.  Assume that for~all $i$, $x_i\ne\xhat$,
and denote $\rho(x_i)$ by~$\rho_i$.  For~all $i$, there is an angle
$\psi_i$ and a unit length vector $u_i$ defined by the equation $x_i
=\xhat\cos\psi_i +u_i\sin\psi_i$, such that $\xhat^\T u_i=0$. By
Algorithm~\ref{al:newtonray} $$\eqalign{x_{i+1} &=\xhat\cos\psi_{i+1}
+u_{i+1}\sin\psi_{i+1} =x_i\cos\theta_i +H_i\sin\theta_i/\theta_i\cr
&=\xhat\biggl({\alpha_i\sin\theta_i\over (\lambda-\rho_i)\theta_i}
+\beta_i\biggr)\cos\psi_i +\biggl( {\alpha_i\sin\theta_i\over\theta_i}
(Q-\rho_iI)^{-1}u_i +\beta_iu_i\biggr)\sin\psi_i,\cr}$$ where
$\beta_i=\cos\theta_i-\sin\theta_i/\theta_i$.  Therefore,
\begin{equation}\label{eq:tanpsi} |\tan\psi_{i+1}| ={\Bigl\|
{\alpha_i\sin\theta_i\over\theta_i} (Q-\rho_iI)^{-1}u_i
+\beta_iu_i\Bigr\|\over \Bigl| {\alpha_i\sin\theta_i\over
(\lambda-\rho_i)\theta_i} +\beta_i\Bigr|} \cdot|\tan\psi_i|.
\end{equation} The following equalities and low order approximations
in terms of the small quantities $\lambda-\rho_i$, $\theta_i$, and
$\psi_i$ are straightforward to establish: ${\lambda-\rho_i}
={(\lambda-\rho(u_i))}
\discretionary{}{\the\textfont2\char2\thinspace}{}
\sin^2\psi_i$, $\theta_i^2=\cos^2\psi_i\sin^2\psi_i +\hot$, 
$\alpha_i={(\lambda-\rho_i)} +\hot$, and $\beta_i=-\theta_i^2/3
+\hot$\spacefactor=3000\relax\space Thus, the denominator of the large
fraction in Eq.~(\ref{eq:tanpsi}) is of order unity and the numerator
is of order $\sin^2\psi_i$.  Therefore, we have $$|\psi_{i+1}|={\rm
const.}\,|\psi_i|^3 +\hot\tombstone$$
\end{proof}

\begin{remark} If Algorithm~\ref{al:newtonray} is simplified by
replacing Step~2 with
\begin{steps}
\step[2.$\!'$] Compute $$x_{i+1}=y_i\big/\|y_i\|,$$ increment
     $i$, and go~to Step~1.
\end{steps}
then we obtain the Rayleigh quotient iteration.  These two algorithms
differ by the method in which they use the vector
$y_i=(Q-\rho(x_i)I)^{-1}x_i$ to compute the next iterate on the
sphere.  Algorithm~\ref{al:newtonray} computes the point $H_i$
in~$T_{x_i}S^{n-1}$ where $y_i$ intersects this tangent plane, then
computes $x_{i+1}$ via the exponential map of this vector (which
``rolls'' the tangent vector $H_i$ onto the sphere).  The Rayleigh
quotient iteration computes the intersection of~$y_i$ with the sphere
itself and takes this intersection to be $x_{i+1}$. The latter
approach approximates Algorithm~\ref{al:newtonray} up to quadratic
terms when $x_i$ is close to an eigenvector.
Algorithm~\ref{al:newtonray} is more expensive to compute
than---though of the same order as---the Rayleigh quotient iteration;
thus, the |RQI| is seen to be an efficient approximation of Newton's
method.
\end{remark}

If the exponential map is replaced by the chart $v\in T_x\mapsto
(x+v)/\|x+v\|\in S^{n-1}$, Shub~\cite{Shub:ray} shows that a
corresponding version of Newton's method is equivalent to the |RQI|.

\begin{example}[The function $\lyapunov$]\label{eg:trHNnewton}\ignorespaces
Let $\Theta$, $Q$, $H=\Ad_{\Theta^\T}(Q)$, and $\Omega$ be as in
Example~\ref{eg:trHNgrad}.  The second covariant differential of
$f(\Theta)=\lyapunov$ may be computed either by polarization of the
second order term of $\tr\Ad_{e^{-t\Omega}}(H)N$, or by covariant
differentiation of the differential $df_\Theta
=-\tr[H,N]\Theta^\T(\rdot)$: $$(\nabla^2\f)_\Theta(\Theta X,\Theta Y)
=-\half\tr\bigl([H,\ad_X N]-[\ad_X H,N]\bigr)Y,$$ where $X$,
$Y\in\so(n)$.  To compute the direction $\Theta X\in T_\Theta$,
$X\in\so(n)$, for Newton's method, we must solve the equation
$(\nabla^2\f)_\Theta(\Theta\rdot,\Theta X)=df_\Theta$, which yields
the linear equation $$L_\Theta(X)\buildrel{\rm def}\over=[H,\ad_X
N]-[\ad_X H,N] =2[H,N].$$ The linear operator
$L_\Theta\colon\so(n)\to\so(n)$ is self-adjoint for~all $\Theta$ and,
in a neighborhood of the maximum, negative definite.  Therefore,
standard iterative techniques in the vector space $\so(n)$, such as
the classical conjugate gradient method, may be used to solve this
equation near the maximum.  The results of a numerical experiment
demonstrating the convergence of Newton's method in~$\SO(20)$ are
shown in Figure~\ref{fig:SOnconv}.  As can be seen, Newton's method
converged within round-off error in two iterations.
\end{example}

\begin{remark}\label{rem:trHNcub}\ignorespaces If Newton's method
applied to the function $f(\Theta)=\lyapunov$ converges to the point
$\Thetahat$ such that $\Ad_{\Thetahat^\T}(Q)=H_\infty=\alpha N$,
$\alpha\in\R$, then it converges cubically.
\end{remark}

\begin{proof} By covariant differentiation of~$\nabla^2\f\!$, the third
covariant differential of~$f$ at~$\Theta$ evaluated at the tangent
vectors $\Theta X$, $\Theta Y$, $\Theta Z\in T_\Theta$, $X$, $Y$,
$Z\in\so(n)$, is $$\eqalign{(\nabla^3\f)_\Theta(\Theta X,\Theta
Y,\Theta Z) =-\quarter\tr\bigl(&[\ad_Y\ad_ZH,N] -[\ad_Z\ad_YN,H]\cr
{}+[H,\ad_{\ad_YZ}N]-{}&[\ad_YH,\ad_ZN] +[\ad_YN,\ad_ZH]\bigr)
X.\cr}$$ If $H=\alpha N$, $\alpha\in\R$, then $(\nabla^3\f)
_\Theta(\rdot,\Theta X,\Theta X)=0$.  Therefore, the second order
terms on the right hand side of Eq.~(\ref{eq:quadproof3}) vanish at
the critical point.  The remark follows from the smoothness of~$f$.
\end{proof}

This remark illuminates how rapid convergence of Newton's method
applied to the function $f$ can be achieved in some instances.  If
$E_{ij}\in\so(n)$ is a matrix with entry $+1$ at element $(i,j)$, $-1$
at element $(j,i)$, and zero elsewhere, $X=\sum_{i<j}x^{ij}E_{ij}$,
$H=\diag(h_{1},\ldots,h_{n})$, and $N=\diag(\nu_1,\ldots,\nu_n)$, then
$$\eqalign{&(\nabla^3\f)_\Theta(\Theta E_{ij},\Theta X,\Theta X)={}\cr
&\qquad{-2}\sum_{k\neq i,j}x^{ik}x^{jk}\bigl((h_i\nu_j-h_j\nu_i)
+(h_j\nu_k-h_k\nu_j) +(h_k\nu_i-h_i\nu_k)\bigr).\cr}$$ If the $h_i$
are close to $\alpha \nu_i$, $\alpha\in\R$, for~all~$i$, then
$(\nabla^3\f)_\Theta(\rdot,\Theta X,\Theta X)$ may be small, yielding
a fast rate of quadratic convergence.

\begin{example}[Jacobi's method]\label{eg:jacobi}\ignorespaces Let
$\pi$ be the projection of a square matrix onto its diagonal, and let
$Q$ be as above.  Consider the maximization of the function
$f(\Theta)=\tr H\pi(H)$, $H=\Ad_{\Theta^\T}(Q)$, on the special
orthogonal group. This is equivalent to minimizing the sum of the
squares of the off-diagonal elements of~$H$ (Golub and
Van~Loan~\cite{GVL} derive the classical Jacobi method).  The gradient
of this function at~$\Theta$ is $2\Theta[H,\pi(H)]$~\cite{Chu:grad}.
By repeated covariant differentiation of~$f\!$, we find
$$\def\.{\kern-30pt}\eqalign{(\nabla\f)_I(X) &=-2\tr[H,\pi(H)]X\cr
(\nabla^2\f)_I(X,Y) &=-\tr\bigl([H,\ad_X\pi(H)] -[\ad_XH,\pi(H)]
-2[H,\pi(\ad_XH)]\bigr)Y\cr (\nabla^3\f)_I(X,Y,Z)
&=-\half\tr\bigl([\ad_Y\ad_ZH,\pi(H)] -[\ad_Z\ad_Y\pi(H),H]\cr
&\.{}+[H,\ad_{\ad_YZ}\pi(H)]-[\ad_YH,\ad_Z\pi(H)]+[\ad_Y\pi(H),\ad_ZH]\cr
&\.{}+2[H,\pi(\ad_Y\ad_ZH)]+2[H,\pi(\ad_Z\ad_YH)]\cr
&\.{}+2[\ad_YH,\pi(\ad_ZH)]-2[H,\ad_Y\pi(\ad_ZH)]\cr
&\.{}+2[\ad_ZH,\pi(\ad_YH)]-2[H,\ad_Z\pi(\ad_YH)]\bigr)X\cr}$$ where
$I$ is the identity matrix and $X$, $Y$, $Z\in\so(n)$.  It is easily
shown that if $[H,\pi(H)]=0$, i.e., if $H$ is diagonal, then
$(\nabla^3\f)_\Theta(\rdot,\Theta X,\Theta X)=0$ (n.b.\
$\pi(\ad_XH)=0$). Therefore, by the same argument as the proof of
Remark~\ref{rem:trHNcub}, Newton's method applied to the function~$\tr
H\pi(H)$ converges cubically.
\end{example}

\section{\spaceskip=.3333em minus.15em  
Conjugate gradient on Riemannian manifolds}
\label{sec:conjgrad}

The method of steepest descent provides an optimization technique
which is relatively inexpensive per iteration, but converges
relatively slowly.  Each step requires the computation of a geodesic
and a gradient direction.  Newton's method provides a technique which
is more costly both in terms of computational complexity and memory
requirements, but converges relatively rapidly.  Each step requires
the computation of a geodesic, a gradient, a second covariant
differential, and its inverse.  In this section we describe the
conjugate gradient method, which has the dual advantages of
algorithmic simplicity and superlinear convergence.

Hestenes and Stiefel~\cite{HestenesStiefel} first used conjugate
gradient methods to compute the solutions of linear equations, or,
equivalently, to compute the minimum of a quadratic form on~$\R^n$.
This approach can be modified to yield effective algorithms to compute
the minima of nonquadratic functions on~$\R^n$.  In particular,
Fletcher and Reeves~\cite{FletcherReeves} and Polak and
Ribi\`ere~\cite{Polak} provide algorithms based upon the assumption
that the second order Taylor expansion of the function to be minimized
sufficiently approximates this function near the minimum.  In
addition, Davidon, Fletcher, and Reeves developed the variable metric
methods~\cite{Fletcher,Polak}, but these will not be discussed here.
One noteworthy feature of conjugate gradient algorithms on~$\R^n$ is
that when the function in question is quadratic, they compute its
minimum in no more than $n$ steps.

The conjugate gradient method on Euclidean space is uncomplicated.
Given a function $f\colon\R^n\to\R$ with continuous second derivatives
and a local minimum at~$\xhat$, and an initial point $x_0\in\R^n$, the
algorithm is initialized by computing the (negative) gradient
direction $G_0=H_0=-(\grad\f)_{x_0}$.  The recursive part of the
algorithm involves (i)~a line minimization of~$f$ along the affine
space $x_i+tH_i$, $t\in\R$, where the minimum occurs at, say,
$t=\lambda_i$, (ii)~computation of the step
$x_{i+1}=x_i+\lambda_iH_i$, (iii)~computation of the (negative)
gradient $G_{i+1}=-(\grad\f)_{x_{i+1}}$, and (iv)~computation of the
next direction for line minimization, \begin{equation}\label{eq:EucCG}
H_{i+1}=G_{i+1}+\gamma_iH_i, \end{equation} where $\gamma_i$ is chosen
such that $H_i$ and $H_{i+1}$ conjugate with respect to the Hessian
matrix of~$f$ at~$\xhat$.  When $f$ is a quadratic form represented by
the symmetric positive definite matrix~$Q$, the conjugacy condition
becomes $H_i^\T QH_{i+1}=0$; therefore, $\gamma_i=-H_i^\T
QG_{i+1}/H_i^\T QH_i$.  It can be shown in this case that the sequence
of vectors $G_i$ are all mutually orthogonal and the sequence of
vectors $H_i$ are all mutually conjugate with respect to~$Q$.  Using
these facts, the computation of $\gamma_i$ may be simplified with the
observation that $\gamma_i =\|G_{i+1}\|^2/\|G_i\|^2$ (Fletcher-Reeves)
or $\gamma_i =(G_{i+1}-G_i)^\T G_{i+1}/\|G_i\|^2$ (Polak-Ribi\`ere).
When $f$ is not quadratic, it is assumed that its second order Taylor
expansion sufficiently approximates $f$ in a neighborhood of the
minimum, and the $\gamma_i$ are chosen so that $H_i$ and $H_{i+1}$ are
conjugate with respect to the matrix $(\partial^2\f/\partial
x^i\partial x^j)(x_{i+1})$ of second partial derivatives of~$f$
at~$x_{i+1}$.  It may be desirable to ``reset'' the algorithm by
setting $H_{i+1}=G_{i+1}$ every $r$th step (frequently, $r=n$) because
the conjugate gradient method does not, in general, converge in $n$
steps if the function $f$ is nonquadratic.  However, if $f$ is closely
approximated by a quadratic function, the reset strategy may be
expected to converge rapidly, whereas the unmodified algorithm may not
be.

Many of these ideas have straightforward generalizations in the
geometry of Riemannian manifolds; several of them have already
appeared.  We need only make the following definition.

\begin{definition} Given a tensor field $\omega$ of type~$(0,2)$ on~$M$
such that for $p$ in~$M$, $\omega_p\colon T_p\times T_p\to\R$ is a
symmetric bilinear form, the tangent vectors $X$ and $Y$ in~$T_p$ are
said to be {\it $\omega_p$-conjugate\/} or {\it conjugate with respect
to $\omega_p$} if $\omega_p(X,Y)=0$.
\end{definition}

An outline of the conjugate gradient method on Riemannian manifolds
may now be given.  Let $M$ be an $n$-dimensional Riemannian manifold
with Riemannian structure $g$ and Levi-Civita connection $\nabla$,
and let $f\in C^\infty(M)$ have a local minimum at~$\phat$.  As in the
conjugate gradient method on Euclidean space, choose an initial point
$p_0$ in~$M$ and compute the (negative) gradient directions
$G_0=H_0=-(\grad\f)_{p_0}$ in~$T_{p_0}$.  The recursive part of the
algorithm involves minimizing $f$ along the geodesic
$t\mapsto\exp_{p_i}tH_i$, $t\in\R$, making a step along the geodesic
to the minimum point $p_{i+1}=\exp\lambda_iH_i$, computing
$G_{i+1}=-(\grad\f)_{p_{i+1}}$, and computing the next direction in
$T_{p_{i+1}}$ for geodesic minimization. This direction is given by
the formula \begin{equation}\label{eq:RieCG}
H_{i+1}=G_{i+1}+\gamma_i\tau H_i, \end{equation} where $\tau$ is the
parallel translation with respect to the geodesic step from~$p_i$
to~$p_{i+1}$, and $\gamma_i$ is chosen such that $\tau H_i$ and
$H_{i+1}$ are $(\nabla^2\f)_{p_{i+1}}$-conjugate, i.e.,
\begin{equation}\label{eq:gammadef} \gamma_i =-{(\nabla^2\f)_{p_{i+1}}
(\tau H_i,G_{i+1})\over (\nabla^2\f)_{p_{i+1}}(\tau H_i,\tau H_i)}.
\end{equation}

Eq.~(\ref{eq:gammadef}) is, in general, expensive to use because the
second covariant differential of~$f$ appears.  However, we can use the
Taylor expansion of~$df$ about $p_{i+1}$ to compute an efficient
approximation of~$\gamma_i$.  By the fact that
$p_i=\exp_{p_{i+1}}(-\lambda_i\tau H_i)$ and by Eq.~(\ref{eq:taylort}),
we have $$\tau df_{p_i} =\tau df_{\exp_{p_{i+1}}(-\lambda_i\tau H_i)}
=df_{p_{i+1}} -\lambda_i(\nabla^2\f)_{p_{i+1}}(\rdot,\tau H_i) +\hot$$
Therefore, the numerator of the right hand side of
Eq.~(\ref{eq:gammadef}) multiplied by the step size $\lambda_i$ can be
approximated by the equation $$\eqalign{\lambda_i
(\nabla^2\f)_{p_{i+1}} (\tau H_i, G_{i+1}) &=df_{p_{i+1}}(G_{i+1})
-(\tau df_{p_i})(G_{i+1})\cr &=-\(G_{i+1} -\tau G_i, G_{i+1}\)\cr}$$
because, by definition, $G_i=-(\grad\f)_{p_i}$, $i=0$, $1$,~\dots, and
for any $X$ in~$T_{p_{i+1}}$, $(\tau df_{p_i})(X)
=df_{p_i}(\tau^{-1}X) =\((\grad\f)_{p_i}, \tau^{-1}X\)
=\(\tau(\grad\f)_{p_i}, X\)$. Similarly, the denominator of the right
hand side of Eq.~(\ref{eq:gammadef}) multiplied by $\lambda_i$ can be
approximated by the equation $$\eqalign{\lambda_i
(\nabla^2\f)_{p_{i+1}} (\tau H_i,\tau H_i) &=df_{p_{i+1}}(\tau H_i)
-(\tau df_{p_i})(\tau H_i)\cr &=\(G_i,H_i\)\cr}$$ because
$\(G_{i+1},\tau H_i\)=0$ by the assumption that $f$ is minimized along
the geodesic $t\mapsto\exp tH_i$ at~$t=\lambda_i$.  Combining these
two approximations with Eq.~(\ref{eq:gammadef}), we obtain a formula
for~$\gamma_i$ that is relatively inexpensive to compute:
\begin{equation}\label{eq:gammaeff} \gamma_i ={\(G_{i+1} -\tau
G_i,G_{i+1}\)\over \(G_i,H_i\)}. \end{equation} Of course, as the
connection $\nabla$ is compatible with the metric $g$, the denominator
of Eq.~(\ref{eq:gammaeff}) may be replaced, if desired, by $\(\tau
G_i,\tau H_i\)$.

The conjugate gradient method may now be presented in full.

\begin{algorithm}[Conjugate gradient
method]\label{al:cgman}\ignorespaces Let $M$ be a complete Riemannian
manifold with Riemannian structure $g$ and Levi-Civita connection
$\nabla$, and let $f$ be a $C^\infty$ function on~$M$.
\begin{steps}
\step[0.] Select $p_0\in M$,
     compute $G_0=H_0=-(\grad\f)_{p_0}$, and set $i=0$.
\step[1.] Compute $\lambda_i$ such that
     $$f(\exp_{p_i}\lambda_iH_i)\le f(\exp_{p_i}\lambda H_i)$$ for all
     $\lambda\ge0$.
\step[2.] Set $p_{i+1}=\exp_{p_i}\lambda_iH_i$.
\step[3.] Set $$\eqalign{G_{i+1}
     &=-(\grad\f)_{p_{i+1}},\cr H_{i+1} &=G_{i+1} +\gamma_i\tau H_i,
     \qquad\gamma_i=\smash{{\(G_{i+1}-\tau G_i,G_{i+1}\)\over
     \(G_i,H_i\)}},\cr \noalign{\vskip2pt}}$$ where $\tau$ is the
     parallel translation with respect to the geodesic from~$p_i$
     to~$p_{i+1}$.  If $i\equiv n-1\ (\bmod\ n)$, set
     $H_{i+1}=G_{i+1}$.  Increment $i$, and go~to Step~1.
\end{steps}
\end{algorithm}

\begin{theorem}\label{th:cgsuper}\ignorespaces Let $f\in C^\infty(M)$
have a nondegenerate critical point at $\phat$ such that the Hessian
$(d^2\f)_\phat$ is positive definite.  Let $p_i$ be a sequence of
points in~$M$ generated by Algorithm~\ref{al:cgman} converging
to~$\phat$.  Then there exists a constant $\theta>0$ and an integer
$N$ such that for~all $i\ge N$,
$$d(p_{i+n},\phat)\le \theta d^2(p_i,\phat).$$
\end{theorem}

Note that linear convergence is already guaranteed by
Theorem~\ref{th:linearconv}.

\begin{proof} If $p_j=\phat$ for some integer $j$, the assertion
becomes trivial; assume otherwise.  Recall that if $X_1$,~\dots, $X_n$
is some basis for $T_\phat$, then the map $\exp_\phat(a^1X_1 +\cdots+
a^nX_n) \buildrel\nu\over\to (a^1,\ldots,a^n)$ defines a set of normal
coordinates at~$\phat$.  Let $N_\phat$ be a normal neighborhood
of~$\phat$ on which the normal coordinates $\nu=(x^1,\ldots,x^n)$ are
defined. Consider the map $\nupushf\buildrel\smash
{\scriptscriptstyle\rm def}\over= f\circ\nu^{-1}\colon\R^n\to\R$.  By
the smoothness of $f$ and $\exp$, $\nupushf$ has a critical point at
$0\in\R^n$ such that the Hessian matrix of~$\nupushf$ at~$0$ is
positive definite.  Indeed, by the fact that $(d\exp)_0=\id$, the
$ij$th component of the Hessian matrix of $\nupushf$ at~$0$ is given
by $(d^2\f)_\phat(X_i,X_j)$.

Therefore, there exists a neighborhood $U$ of $0\in\R^n$, a constant
$\theta'>0$, and an integer $N$, such that for any initial point
$x_0\in U$, the conjugate gradient method on Euclidean space (with
resets) applied to the function $\nupushf$ yields a sequence of points
$x_i$ converging to~$0$ such that for~all $i\ge N$,
$$\|x_{i+n}\|\le\theta'\|x_i\|^2.$$ See Polak~\cite[p.~260ff\/]{Polak}
for a proof of this fact.  Let $x_0=\nu(p_0)$ in~$U$ be an initial
point.  Because $\exp$ is not an isometry, Algorithm~\ref{al:cgman}
yields a different sequence of points in~$\R^n$ than the classical
conjugate gradient method on $\R^n$ (upon equating points in a
neighborhood of~$\phat\in M$ with points in a neighborhood
of~$0\in\R^n$ via the normal coordinates).

Nevertheless, the amount by which $\exp$ fails to preserve inner
products can be quantified via the Gauss Lemma and Jacobi's equation;
see, e.g., Cheeger and Ebin~\cite{Cheegin}, or the appendices of
Karcher~\cite{Karcher}.  Let $t$ be small, and let $X\in T_\phat$ and
$Y\in T_{tX}(T_\phat)\cong T_\phat$ be orthonormal tangent vectors.
The amount by which the exponential map changes the length of tangent
vectors is approximated by the Taylor expansion $$\|d\exp(tY)\|^2 =t^2
-\third Kt^4 +\hot$$ where $K$ is the sectional curvature of~$M$ along
the section in~$T_\phat$ spanned by $X$ and $Y$.  Therefore, near
$\phat$ Algorithm~\ref{al:cgman} differs from the conjugate gradient
method on~$\R^n$ applied to the function $\nupushf$ only by third
order and higher terms.  Thus both algorithms have the same rate of
convergence.  The theorem follows.
\end{proof}

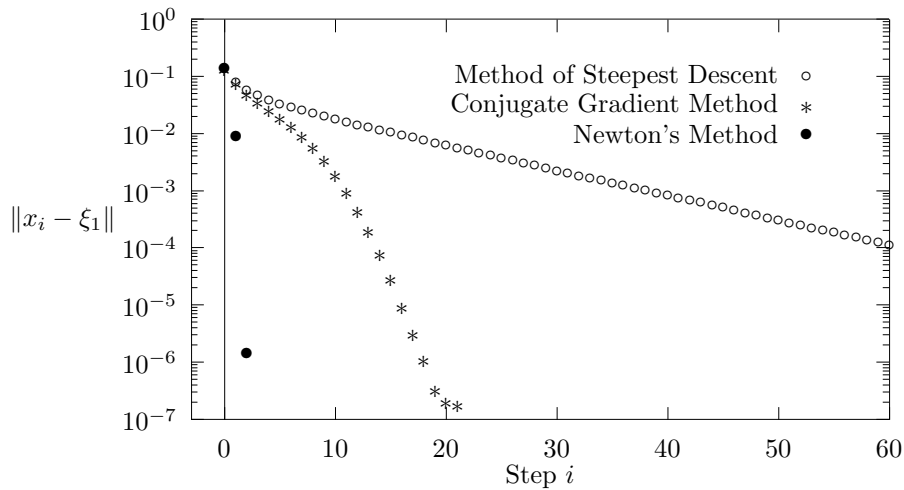
\begin{figure}
\vskip-.25in
\begingroup  
\let\p=\put
\let\r=\rule
\def\c{\circle{12}}
\def\D{\raisebox{-1.2pt}{\makebox(0,0){$\Diamond$}}}
\def\s{\circle*{18}}
\setlength{\unitlength}{0.240900pt}
\ifx\plotpoint\undefined\newsavebox{\plotpoint}\fi
\begin{picture}(1424,900)(0,0)
\tenrm
\sbox{\plotpoint}{\r[-0.175pt]{0.350pt}{0.350pt}}%
\p(316,158){\r[-0.175pt]{0.350pt}{151.526pt}}
\p(264,158){\r[-0.175pt]{4.818pt}{0.350pt}}
\p(242,158){\makebox(0,0)[r]{$10^{\gnulog 1e-07}$}}
\p(1340,158){\r[-0.175pt]{4.818pt}{0.350pt}}
\p(264,185){\r[-0.175pt]{2.409pt}{0.350pt}}
\p(1350,185){\r[-0.175pt]{2.409pt}{0.350pt}}
\p(264,201){\r[-0.175pt]{2.409pt}{0.350pt}}
\p(1350,201){\r[-0.175pt]{2.409pt}{0.350pt}}
\p(264,212){\r[-0.175pt]{2.409pt}{0.350pt}}
\p(1350,212){\r[-0.175pt]{2.409pt}{0.350pt}}
\p(264,221){\r[-0.175pt]{2.409pt}{0.350pt}}
\p(1350,221){\r[-0.175pt]{2.409pt}{0.350pt}}
\p(264,228){\r[-0.175pt]{2.409pt}{0.350pt}}
\p(1350,228){\r[-0.175pt]{2.409pt}{0.350pt}}
\p(264,234){\r[-0.175pt]{2.409pt}{0.350pt}}
\p(1350,234){\r[-0.175pt]{2.409pt}{0.350pt}}
\p(264,239){\r[-0.175pt]{2.409pt}{0.350pt}}
\p(1350,239){\r[-0.175pt]{2.409pt}{0.350pt}}
\p(264,244){\r[-0.175pt]{2.409pt}{0.350pt}}
\p(1350,244){\r[-0.175pt]{2.409pt}{0.350pt}}
\p(264,248){\r[-0.175pt]{4.818pt}{0.350pt}}
\p(242,248){\makebox(0,0)[r]{$10^{\gnulog 1e-06}$}}
\p(1340,248){\r[-0.175pt]{4.818pt}{0.350pt}}
\p(264,275){\r[-0.175pt]{2.409pt}{0.350pt}}
\p(1350,275){\r[-0.175pt]{2.409pt}{0.350pt}}
\p(264,291){\r[-0.175pt]{2.409pt}{0.350pt}}
\p(1350,291){\r[-0.175pt]{2.409pt}{0.350pt}}
\p(264,302){\r[-0.175pt]{2.409pt}{0.350pt}}
\p(1350,302){\r[-0.175pt]{2.409pt}{0.350pt}}
\p(264,311){\r[-0.175pt]{2.409pt}{0.350pt}}
\p(1350,311){\r[-0.175pt]{2.409pt}{0.350pt}}
\p(264,318){\r[-0.175pt]{2.409pt}{0.350pt}}
\p(1350,318){\r[-0.175pt]{2.409pt}{0.350pt}}
\p(264,324){\r[-0.175pt]{2.409pt}{0.350pt}}
\p(1350,324){\r[-0.175pt]{2.409pt}{0.350pt}}
\p(264,329){\r[-0.175pt]{2.409pt}{0.350pt}}
\p(1350,329){\r[-0.175pt]{2.409pt}{0.350pt}}
\p(264,334){\r[-0.175pt]{2.409pt}{0.350pt}}
\p(1350,334){\r[-0.175pt]{2.409pt}{0.350pt}}
\p(264,338){\r[-0.175pt]{4.818pt}{0.350pt}}
\p(242,338){\makebox(0,0)[r]{$10^{\gnulog 1e-05}$}}
\p(1340,338){\r[-0.175pt]{4.818pt}{0.350pt}}
\p(264,365){\r[-0.175pt]{2.409pt}{0.350pt}}
\p(1350,365){\r[-0.175pt]{2.409pt}{0.350pt}}
\p(264,381){\r[-0.175pt]{2.409pt}{0.350pt}}
\p(1350,381){\r[-0.175pt]{2.409pt}{0.350pt}}
\p(264,392){\r[-0.175pt]{2.409pt}{0.350pt}}
\p(1350,392){\r[-0.175pt]{2.409pt}{0.350pt}}
\p(264,401){\r[-0.175pt]{2.409pt}{0.350pt}}
\p(1350,401){\r[-0.175pt]{2.409pt}{0.350pt}}
\p(264,408){\r[-0.175pt]{2.409pt}{0.350pt}}
\p(1350,408){\r[-0.175pt]{2.409pt}{0.350pt}}
\p(264,414){\r[-0.175pt]{2.409pt}{0.350pt}}
\p(1350,414){\r[-0.175pt]{2.409pt}{0.350pt}}
\p(264,419){\r[-0.175pt]{2.409pt}{0.350pt}}
\p(1350,419){\r[-0.175pt]{2.409pt}{0.350pt}}
\p(264,423){\r[-0.175pt]{2.409pt}{0.350pt}}
\p(1350,423){\r[-0.175pt]{2.409pt}{0.350pt}}
\p(264,428){\r[-0.175pt]{4.818pt}{0.350pt}}
\p(242,428){\makebox(0,0)[r]{$10^{\gnulog 1e-04}$}}
\p(1340,428){\r[-0.175pt]{4.818pt}{0.350pt}}
\p(264,455){\r[-0.175pt]{2.409pt}{0.350pt}}
\p(1350,455){\r[-0.175pt]{2.409pt}{0.350pt}}
\p(264,470){\r[-0.175pt]{2.409pt}{0.350pt}}
\p(1350,470){\r[-0.175pt]{2.409pt}{0.350pt}}
\p(264,482){\r[-0.175pt]{2.409pt}{0.350pt}}
\p(1350,482){\r[-0.175pt]{2.409pt}{0.350pt}}
\p(264,490){\r[-0.175pt]{2.409pt}{0.350pt}}
\p(1350,490){\r[-0.175pt]{2.409pt}{0.350pt}}
\p(264,497){\r[-0.175pt]{2.409pt}{0.350pt}}
\p(1350,497){\r[-0.175pt]{2.409pt}{0.350pt}}
\p(264,504){\r[-0.175pt]{2.409pt}{0.350pt}}
\p(1350,504){\r[-0.175pt]{2.409pt}{0.350pt}}
\p(264,509){\r[-0.175pt]{2.409pt}{0.350pt}}
\p(1350,509){\r[-0.175pt]{2.409pt}{0.350pt}}
\p(264,513){\r[-0.175pt]{2.409pt}{0.350pt}}
\p(1350,513){\r[-0.175pt]{2.409pt}{0.350pt}}
\p(264,517){\r[-0.175pt]{4.818pt}{0.350pt}}
\p(242,517){\makebox(0,0)[r]{$10^{\gnulog 1e-03}$}}
\p(1340,517){\r[-0.175pt]{4.818pt}{0.350pt}}
\p(264,544){\r[-0.175pt]{2.409pt}{0.350pt}}
\p(1350,544){\r[-0.175pt]{2.409pt}{0.350pt}}
\p(264,560){\r[-0.175pt]{2.409pt}{0.350pt}}
\p(1350,560){\r[-0.175pt]{2.409pt}{0.350pt}}
\p(264,572){\r[-0.175pt]{2.409pt}{0.350pt}}
\p(1350,572){\r[-0.175pt]{2.409pt}{0.350pt}}
\p(264,580){\r[-0.175pt]{2.409pt}{0.350pt}}
\p(1350,580){\r[-0.175pt]{2.409pt}{0.350pt}}
\p(264,587){\r[-0.175pt]{2.409pt}{0.350pt}}
\p(1350,587){\r[-0.175pt]{2.409pt}{0.350pt}}
\p(264,593){\r[-0.175pt]{2.409pt}{0.350pt}}
\p(1350,593){\r[-0.175pt]{2.409pt}{0.350pt}}
\p(264,599){\r[-0.175pt]{2.409pt}{0.350pt}}
\p(1350,599){\r[-0.175pt]{2.409pt}{0.350pt}}
\p(264,603){\r[-0.175pt]{2.409pt}{0.350pt}}
\p(1350,603){\r[-0.175pt]{2.409pt}{0.350pt}}
\p(264,607){\r[-0.175pt]{4.818pt}{0.350pt}}
\p(242,607){\makebox(0,0)[r]{$10^{\gnulog 1e-02}$}}
\p(1340,607){\r[-0.175pt]{4.818pt}{0.350pt}}
\p(264,634){\r[-0.175pt]{2.409pt}{0.350pt}}
\p(1350,634){\r[-0.175pt]{2.409pt}{0.350pt}}
\p(264,650){\r[-0.175pt]{2.409pt}{0.350pt}}
\p(1350,650){\r[-0.175pt]{2.409pt}{0.350pt}}
\p(264,661){\r[-0.175pt]{2.409pt}{0.350pt}}
\p(1350,661){\r[-0.175pt]{2.409pt}{0.350pt}}
\p(264,670){\r[-0.175pt]{2.409pt}{0.350pt}}
\p(1350,670){\r[-0.175pt]{2.409pt}{0.350pt}}
\p(264,677){\r[-0.175pt]{2.409pt}{0.350pt}}
\p(1350,677){\r[-0.175pt]{2.409pt}{0.350pt}}
\p(264,683){\r[-0.175pt]{2.409pt}{0.350pt}}
\p(1350,683){\r[-0.175pt]{2.409pt}{0.350pt}}
\p(264,688){\r[-0.175pt]{2.409pt}{0.350pt}}
\p(1350,688){\r[-0.175pt]{2.409pt}{0.350pt}}
\p(264,693){\r[-0.175pt]{2.409pt}{0.350pt}}
\p(1350,693){\r[-0.175pt]{2.409pt}{0.350pt}}
\p(264,697){\r[-0.175pt]{4.818pt}{0.350pt}}
\p(242,697){\makebox(0,0)[r]{$10^{\gnulog 1e-01}$}}
\p(1340,697){\r[-0.175pt]{4.818pt}{0.350pt}}
\p(264,724){\r[-0.175pt]{2.409pt}{0.350pt}}
\p(1350,724){\r[-0.175pt]{2.409pt}{0.350pt}}
\p(264,740){\r[-0.175pt]{2.409pt}{0.350pt}}
\p(1350,740){\r[-0.175pt]{2.409pt}{0.350pt}}
\p(264,751){\r[-0.175pt]{2.409pt}{0.350pt}}
\p(1350,751){\r[-0.175pt]{2.409pt}{0.350pt}}
\p(264,760){\r[-0.175pt]{2.409pt}{0.350pt}}
\p(1350,760){\r[-0.175pt]{2.409pt}{0.350pt}}
\p(264,767){\r[-0.175pt]{2.409pt}{0.350pt}}
\p(1350,767){\r[-0.175pt]{2.409pt}{0.350pt}}
\p(264,773){\r[-0.175pt]{2.409pt}{0.350pt}}
\p(1350,773){\r[-0.175pt]{2.409pt}{0.350pt}}
\p(264,778){\r[-0.175pt]{2.409pt}{0.350pt}}
\p(1350,778){\r[-0.175pt]{2.409pt}{0.350pt}}
\p(264,783){\r[-0.175pt]{2.409pt}{0.350pt}}
\p(1350,783){\r[-0.175pt]{2.409pt}{0.350pt}}
\p(264,787){\r[-0.175pt]{4.818pt}{0.350pt}}
\p(242,787){\makebox(0,0)[r]{$10^{\gnulog 1e+00}$}}
\p(1340,787){\r[-0.175pt]{4.818pt}{0.350pt}}
\p(316,158){\r[-0.175pt]{0.350pt}{4.818pt}} \p(316,113){\makebox(0,0){$0$}}
\p(316,767){\r[-0.175pt]{0.350pt}{4.818pt}}
\p(490,158){\r[-0.175pt]{0.350pt}{4.818pt}} \p(490,113){\makebox(0,0){$10$}}
\p(490,767){\r[-0.175pt]{0.350pt}{4.818pt}}
\p(664,158){\r[-0.175pt]{0.350pt}{4.818pt}} \p(664,113){\makebox(0,0){$20$}}
\p(664,767){\r[-0.175pt]{0.350pt}{4.818pt}}
\p(838,158){\r[-0.175pt]{0.350pt}{4.818pt}} \p(838,113){\makebox(0,0){$30$}}
\p(838,767){\r[-0.175pt]{0.350pt}{4.818pt}}
\p(1012,158){\r[-0.175pt]{0.350pt}{4.818pt}} \p(1012,113){\makebox(0,0){$40$}}
\p(1012,767){\r[-0.175pt]{0.350pt}{4.818pt}}
\p(1186,158){\r[-0.175pt]{0.350pt}{4.818pt}} \p(1186,113){\makebox(0,0){$50$}}
\p(1186,767){\r[-0.175pt]{0.350pt}{4.818pt}}
\p(1360,158){\r[-0.175pt]{0.350pt}{4.818pt}} \p(1360,113){\makebox(0,0){$60$}}
\p(1360,767){\r[-0.175pt]{0.350pt}{4.818pt}}
\p(264,158){\r[-0.175pt]{264.026pt}{0.350pt}}
\p(1360,158){\r[-0.175pt]{0.350pt}{151.526pt}}
\p(264,787){\r[-0.175pt]{264.026pt}{0.350pt}}
\p(-21,472){\makebox(0,0)[l]{\shortstack{$\|x_i-\xi_1\|$}}}
\p(812,68){\makebox(0,0){Step $i$}}
\p(264,158){\r[-0.175pt]{0.350pt}{151.526pt}}
\p(1186,697){\makebox(0,0)[r]{Method of Steepest Descent}}
\p(1230,697){\c} \p(316,710){\c} \p(334,688){\c} \p(351,676){\c}
\p(368,667){\c} \p(386,660){\c} \p(403,654){\c} \p(421,649){\c}
\p(438,644){\c} \p(455,639){\c} \p(473,634){\c} \p(490,630){\c}
\p(508,625){\c} \p(525,621){\c} \p(542,617){\c} \p(560,613){\c}
\p(577,609){\c} \p(595,605){\c} \p(612,601){\c} \p(629,597){\c}
\p(647,593){\c} \p(664,589){\c} \p(682,585){\c} \p(699,581){\c}
\p(716,577){\c} \p(734,573){\c} \p(751,569){\c} \p(769,565){\c}
\p(786,561){\c} \p(803,557){\c} \p(821,553){\c} \p(838,549){\c}
\p(855,545){\c} \p(873,541){\c} \p(890,538){\c} \p(908,534){\c}
\p(925,530){\c} \p(942,526){\c} \p(960,522){\c} \p(977,518){\c}
\p(995,514){\c} \p(1012,510){\c} \p(1029,506){\c} \p(1047,502){\c}
\p(1064,499){\c} \p(1082,495){\c} \p(1099,491){\c} \p(1116,487){\c}
\p(1134,483){\c} \p(1151,479){\c} \p(1169,475){\c} \p(1186,471){\c}
\p(1203,467){\c} \p(1221,463){\c} \p(1238,459){\c} \p(1256,455){\c}
\p(1273,452){\c} \p(1290,448){\c} \p(1308,444){\c} \p(1325,440){\c}
\p(1343,436){\c} \p(1360,432){\c} \p(1186,652){\makebox(0,0)[r]{Conjugate
Gradient Method}} \p(1230,652){\D} \p(316,710){\D} \p(334,688){\D}
\p(351,672){\D} \p(368,659){\D} \p(386,647){\D} \p(403,634){\D}
\p(421,621){\D} \p(438,605){\D} \p(455,588){\D} \p(473,567){\D}
\p(490,544){\D} \p(508,518){\D} \p(525,488){\D} \p(542,456){\D}
\p(560,420){\D} \p(577,380){\D} \p(595,337){\D} \p(612,295){\D}
\p(629,254){\D} \p(647,207){\D} \p(664,188){\D} \p(682,183){\D}
\p(1186,607){\makebox(0,0)[r]{Newton's Method}} \p(1230,607){\s}
\p(316,710){\s} \p(334,603){\s} \p(351,263){\s}
\end{picture}
\endgroup  
\vskip-.25in
\caption[]{\protect\small\label{fig:raysphere}\ignorespaces
Maximization of Rayleigh's quotient $x^\T Qx$
on~$S^{20}\subset\R^{21}$, where $Q=\diag(21,\ldots,1)$.  The $i$th
iterate is $x_i$, and $\xi_1$ is the eigenvector corresponding to the
largest eigenvalue of~$Q$.  Algorithm~\protect\ref{al:newtonray} was
used for Newton's method and Algorithm~\protect\ref{al:raysphere} was
used for the conjugate gradient method.}
\end{figure}

\begin{example}[Rayleigh's quotient on the
sphere]\label{eg:rayqcg}\ignorespaces Applied to Rayleigh's quotient
on the sphere, the conjugate gradient method provides an efficient
technique to compute the eigenvectors corresponding to the largest or
smallest eigenvalue of a real symmetric matrix.  Let $S^{n-1}$ and
$\rho(x)=x^\T Qx$ be as in Examples \ref{eg:rayqgrad} and
\ref{eg:newtonray}.  From Algorithm~\ref{al:cgman}, we have the
following algorithm.

\begin{algorithm}[|CG| method for the extreme
eigenvalue/eigenvector]\label{al:raysphere}\ignorespaces Let $Q$ be a
real symmetric \hbox{$n$-by-$n$} matrix.
\begin{steps}
\step[0.] Select $x_0$ in~$\R^n$ such that $x_0^\T x_0=1$, compute
    $G_0=H_0=(Q-\rho(x_0)I)x_0$, and set $i=0$.
\step[1.] Compute $c$, $s$, and $v=1-c=s^2/(1+c)$, such that
    $\rho(x_ic+h_is)$ is maximized, where $c^2+s^2=1$ and
    $h_i=H_i/\|H_i\|$.  This can be accomplished by geodesic
    minimization, or by the formulae $$\eqalign{c
    &=\bigl(\half(1+b/r)\bigr)^\half\cr s &=a/(2rc)\cr} \quad\hbox{if
    $b\geq0$,} \quad\hbox{or}\quad \eqalign{s
    &=\bigl(\half(1-b/r)\bigr)^\half\cr c &=a/(2rs)\cr} \quad\hbox{if
    $b\leq0$,}$$ where $a=2x_i^\T Qh_i$, $b=x_i^\T Qx_i -h_i^\T Qh_i$,
    and $r=\surd(a^2+b^2)$.
\step[2.] Set $$x_{i+1}=x_ic+h_is,\quad \tau H_i=H_ic -
    x_i\|H_i\|s,\quad \tau G_i=G_i -(h_i^\T G_i)(x_is+h_iv).$$
\step[3.] Set $$\eqalign{G_{i+1} &=\bigl(Q-\rho(x_{i+1})I\bigr)x_{i+1},\cr
    H_{i+1} &=G_{i+1} +\gamma_i\tau H_i,
    \qquad\gamma_i=\smash{{(G_{i+1}-\tau G_i)^\T G_{i+1}\over G_i^\T
    H_i}}.\cr\noalign{\vskip2pt}}$$ If $i\equiv n-1\ (\bmod\ n)$, set
    $H_{i+1}=G_{i+1}$.  Increment $i$, and go~to Step~1.
\end{steps}
\end{algorithm}

The convergence rate of this algorithm to the eigenvector
corresponding to the largest eigenvalue of~$Q$ is given by
Theorem~\ref{th:cgsuper}.  This algorithm costs one matrix-vector
multiplication (relatively inexpensive when $Q$ is sparse), one
geodesic minimization or computation of~$\rho(h_i)$, and $10n$~flops
per iteration.  The results of a numerical experiment demonstrating
the convergence of Algorithm~\ref{al:raysphere} on~$S^{20}$ are shown
in Figure~\ref{fig:raysphere}.
\end{example}

Fuhrmann and Liu~\cite{FuhrLiu} provide a conjugate gradient algorithm
for Rayleigh's quotient on the sphere that uses an azimuthal
projection onto tangent planes.

\begin{figure}
\vskip-.25in
\begingroup  
\let\p=\put
\let\r=\rule
\def\c{\circle{12}}
\def\D{\raisebox{-1.2pt}{\makebox(0,0){$\Diamond$}}}
\def\s{\circle*{18}}
\setlength{\unitlength}{0.240900pt}
\ifx\plotpoint\undefined\newsavebox{\plotpoint}\fi
\begin{picture}(1424,900)(0,0)
\tenrm
\sbox{\plotpoint}{\r[-0.175pt]{0.350pt}{0.350pt}}%
\p(333,158){\r[-0.175pt]{0.350pt}{151.526pt}}
\p(264,158){\r[-0.175pt]{4.818pt}{0.350pt}}
\p(242,158){\makebox(0,0)[r]{$10^{\gnulog 1e-05}$}}
\p(1340,158){\r[-0.175pt]{4.818pt}{0.350pt}}
\p(264,190){\r[-0.175pt]{2.409pt}{0.350pt}}
\p(1350,190){\r[-0.175pt]{2.409pt}{0.350pt}}
\p(264,208){\r[-0.175pt]{2.409pt}{0.350pt}}
\p(1350,208){\r[-0.175pt]{2.409pt}{0.350pt}}
\p(264,221){\r[-0.175pt]{2.409pt}{0.350pt}}
\p(1350,221){\r[-0.175pt]{2.409pt}{0.350pt}}
\p(264,231){\r[-0.175pt]{2.409pt}{0.350pt}}
\p(1350,231){\r[-0.175pt]{2.409pt}{0.350pt}}
\p(264,240){\r[-0.175pt]{2.409pt}{0.350pt}}
\p(1350,240){\r[-0.175pt]{2.409pt}{0.350pt}}
\p(264,247){\r[-0.175pt]{2.409pt}{0.350pt}}
\p(1350,247){\r[-0.175pt]{2.409pt}{0.350pt}}
\p(264,253){\r[-0.175pt]{2.409pt}{0.350pt}}
\p(1350,253){\r[-0.175pt]{2.409pt}{0.350pt}}
\p(264,258){\r[-0.175pt]{2.409pt}{0.350pt}}
\p(1350,258){\r[-0.175pt]{2.409pt}{0.350pt}}
\p(264,263){\r[-0.175pt]{4.818pt}{0.350pt}}
\p(242,263){\makebox(0,0)[r]{$10^{\gnulog 1e-04}$}}
\p(1340,263){\r[-0.175pt]{4.818pt}{0.350pt}}
\p(264,294){\r[-0.175pt]{2.409pt}{0.350pt}}
\p(1350,294){\r[-0.175pt]{2.409pt}{0.350pt}}
\p(264,313){\r[-0.175pt]{2.409pt}{0.350pt}}
\p(1350,313){\r[-0.175pt]{2.409pt}{0.350pt}}
\p(264,326){\r[-0.175pt]{2.409pt}{0.350pt}}
\p(1350,326){\r[-0.175pt]{2.409pt}{0.350pt}}
\p(264,336){\r[-0.175pt]{2.409pt}{0.350pt}}
\p(1350,336){\r[-0.175pt]{2.409pt}{0.350pt}}
\p(264,344){\r[-0.175pt]{2.409pt}{0.350pt}}
\p(1350,344){\r[-0.175pt]{2.409pt}{0.350pt}}
\p(264,351){\r[-0.175pt]{2.409pt}{0.350pt}}
\p(1350,351){\r[-0.175pt]{2.409pt}{0.350pt}}
\p(264,358){\r[-0.175pt]{2.409pt}{0.350pt}}
\p(1350,358){\r[-0.175pt]{2.409pt}{0.350pt}}
\p(264,363){\r[-0.175pt]{2.409pt}{0.350pt}}
\p(1350,363){\r[-0.175pt]{2.409pt}{0.350pt}}
\p(264,368){\r[-0.175pt]{4.818pt}{0.350pt}}
\p(242,368){\makebox(0,0)[r]{$10^{\gnulog 1e-03}$}}
\p(1340,368){\r[-0.175pt]{4.818pt}{0.350pt}}
\p(264,399){\r[-0.175pt]{2.409pt}{0.350pt}}
\p(1350,399){\r[-0.175pt]{2.409pt}{0.350pt}}
\p(264,418){\r[-0.175pt]{2.409pt}{0.350pt}}
\p(1350,418){\r[-0.175pt]{2.409pt}{0.350pt}}
\p(264,431){\r[-0.175pt]{2.409pt}{0.350pt}}
\p(1350,431){\r[-0.175pt]{2.409pt}{0.350pt}}
\p(264,441){\r[-0.175pt]{2.409pt}{0.350pt}}
\p(1350,441){\r[-0.175pt]{2.409pt}{0.350pt}}
\p(264,449){\r[-0.175pt]{2.409pt}{0.350pt}}
\p(1350,449){\r[-0.175pt]{2.409pt}{0.350pt}}
\p(264,456){\r[-0.175pt]{2.409pt}{0.350pt}}
\p(1350,456){\r[-0.175pt]{2.409pt}{0.350pt}}
\p(264,462){\r[-0.175pt]{2.409pt}{0.350pt}}
\p(1350,462){\r[-0.175pt]{2.409pt}{0.350pt}}
\p(264,468){\r[-0.175pt]{2.409pt}{0.350pt}}
\p(1350,468){\r[-0.175pt]{2.409pt}{0.350pt}}
\p(264,473){\r[-0.175pt]{4.818pt}{0.350pt}}
\p(242,473){\makebox(0,0)[r]{$10^{\gnulog 1e-02}$}}
\p(1340,473){\r[-0.175pt]{4.818pt}{0.350pt}}
\p(264,504){\r[-0.175pt]{2.409pt}{0.350pt}}
\p(1350,504){\r[-0.175pt]{2.409pt}{0.350pt}}
\p(264,523){\r[-0.175pt]{2.409pt}{0.350pt}}
\p(1350,523){\r[-0.175pt]{2.409pt}{0.350pt}}
\p(264,536){\r[-0.175pt]{2.409pt}{0.350pt}}
\p(1350,536){\r[-0.175pt]{2.409pt}{0.350pt}}
\p(264,546){\r[-0.175pt]{2.409pt}{0.350pt}}
\p(1350,546){\r[-0.175pt]{2.409pt}{0.350pt}}
\p(264,554){\r[-0.175pt]{2.409pt}{0.350pt}}
\p(1350,554){\r[-0.175pt]{2.409pt}{0.350pt}}
\p(264,561){\r[-0.175pt]{2.409pt}{0.350pt}}
\p(1350,561){\r[-0.175pt]{2.409pt}{0.350pt}}
\p(264,567){\r[-0.175pt]{2.409pt}{0.350pt}}
\p(1350,567){\r[-0.175pt]{2.409pt}{0.350pt}}
\p(264,573){\r[-0.175pt]{2.409pt}{0.350pt}}
\p(1350,573){\r[-0.175pt]{2.409pt}{0.350pt}}
\p(264,577){\r[-0.175pt]{4.818pt}{0.350pt}}
\p(242,577){\makebox(0,0)[r]{$10^{\gnulog 1e-01}$}}
\p(1340,577){\r[-0.175pt]{4.818pt}{0.350pt}}
\p(264,609){\r[-0.175pt]{2.409pt}{0.350pt}}
\p(1350,609){\r[-0.175pt]{2.409pt}{0.350pt}}
\p(264,627){\r[-0.175pt]{2.409pt}{0.350pt}}
\p(1350,627){\r[-0.175pt]{2.409pt}{0.350pt}}
\p(264,640){\r[-0.175pt]{2.409pt}{0.350pt}}
\p(1350,640){\r[-0.175pt]{2.409pt}{0.350pt}}
\p(264,651){\r[-0.175pt]{2.409pt}{0.350pt}}
\p(1350,651){\r[-0.175pt]{2.409pt}{0.350pt}}
\p(264,659){\r[-0.175pt]{2.409pt}{0.350pt}}
\p(1350,659){\r[-0.175pt]{2.409pt}{0.350pt}}
\p(264,666){\r[-0.175pt]{2.409pt}{0.350pt}}
\p(1350,666){\r[-0.175pt]{2.409pt}{0.350pt}}
\p(264,672){\r[-0.175pt]{2.409pt}{0.350pt}}
\p(1350,672){\r[-0.175pt]{2.409pt}{0.350pt}}
\p(264,677){\r[-0.175pt]{2.409pt}{0.350pt}}
\p(1350,677){\r[-0.175pt]{2.409pt}{0.350pt}}
\p(264,682){\r[-0.175pt]{4.818pt}{0.350pt}}
\p(242,682){\makebox(0,0)[r]{$10^{\gnulog 1e+00}$}}
\p(1340,682){\r[-0.175pt]{4.818pt}{0.350pt}}
\p(264,714){\r[-0.175pt]{2.409pt}{0.350pt}}
\p(1350,714){\r[-0.175pt]{2.409pt}{0.350pt}}
\p(264,732){\r[-0.175pt]{2.409pt}{0.350pt}}
\p(1350,732){\r[-0.175pt]{2.409pt}{0.350pt}}
\p(264,745){\r[-0.175pt]{2.409pt}{0.350pt}}
\p(1350,745){\r[-0.175pt]{2.409pt}{0.350pt}}
\p(264,755){\r[-0.175pt]{2.409pt}{0.350pt}}
\p(1350,755){\r[-0.175pt]{2.409pt}{0.350pt}}
\p(264,764){\r[-0.175pt]{2.409pt}{0.350pt}}
\p(1350,764){\r[-0.175pt]{2.409pt}{0.350pt}}
\p(264,771){\r[-0.175pt]{2.409pt}{0.350pt}}
\p(1350,771){\r[-0.175pt]{2.409pt}{0.350pt}}
\p(264,777){\r[-0.175pt]{2.409pt}{0.350pt}}
\p(1350,777){\r[-0.175pt]{2.409pt}{0.350pt}}
\p(264,782){\r[-0.175pt]{2.409pt}{0.350pt}}
\p(1350,782){\r[-0.175pt]{2.409pt}{0.350pt}}
\p(264,787){\r[-0.175pt]{4.818pt}{0.350pt}}
\p(242,787){\makebox(0,0)[r]{$10^{\gnulog 1e+01}$}}
\p(1340,787){\r[-0.175pt]{4.818pt}{0.350pt}}
\p(333,158){\r[-0.175pt]{0.350pt}{4.818pt}} \p(333,113){\makebox(0,0){$0$}}
\p(333,767){\r[-0.175pt]{0.350pt}{4.818pt}}
\p(470,158){\r[-0.175pt]{0.350pt}{4.818pt}} \p(470,113){\makebox(0,0){$20$}}
\p(470,767){\r[-0.175pt]{0.350pt}{4.818pt}}
\p(607,158){\r[-0.175pt]{0.350pt}{4.818pt}} \p(607,113){\makebox(0,0){$40$}}
\p(607,767){\r[-0.175pt]{0.350pt}{4.818pt}}
\p(744,158){\r[-0.175pt]{0.350pt}{4.818pt}} \p(744,113){\makebox(0,0){$60$}}
\p(744,767){\r[-0.175pt]{0.350pt}{4.818pt}}
\p(881,158){\r[-0.175pt]{0.350pt}{4.818pt}} \p(881,113){\makebox(0,0){$80$}}
\p(881,767){\r[-0.175pt]{0.350pt}{4.818pt}}
\p(1018,158){\r[-0.175pt]{0.350pt}{4.818pt}} \p(1018,113){\makebox(0,0){$100$}}
\p(1018,767){\r[-0.175pt]{0.350pt}{4.818pt}}
\p(1155,158){\r[-0.175pt]{0.350pt}{4.818pt}} \p(1155,113){\makebox(0,0){$120$}}
\p(1155,767){\r[-0.175pt]{0.350pt}{4.818pt}}
\p(1292,158){\r[-0.175pt]{0.350pt}{4.818pt}} \p(1292,113){\makebox(0,0){$140$}}
\p(1292,767){\r[-0.175pt]{0.350pt}{4.818pt}}
\p(264,158){\r[-0.175pt]{264.026pt}{0.350pt}}
\p(1360,158){\r[-0.175pt]{0.350pt}{151.526pt}}
\p(264,787){\r[-0.175pt]{264.026pt}{0.350pt}}
\p(-43,472){\makebox(0,0)[l]{\shortstack{$\|H_i-D_i\|$}}}
\p(812,68){\makebox(0,0){Step $i$}}
\p(264,158){\r[-0.175pt]{0.350pt}{151.526pt}}
\p(1257,504){\makebox(0,0)[r]{Method of Steepest Descent}}
\p(1301,504){\c} \p(333,712){\c} \p(339,698){\c} \p(346,693){\c}
\p(353,689){\c} \p(360,687){\c} \p(367,685){\c} \p(374,683){\c}
\p(380,682){\c} \p(387,680){\c} \p(394,679){\c} \p(401,678){\c}
\p(408,677){\c} \p(415,676){\c} \p(422,675){\c} \p(428,674){\c}
\p(435,673){\c} \p(442,672){\c} \p(449,671){\c} \p(456,670){\c}
\p(463,670){\c} \p(470,669){\c} \p(476,668){\c} \p(483,668){\c}
\p(490,667){\c} \p(497,666){\c} \p(504,666){\c} \p(511,665){\c}
\p(517,665){\c} \p(524,664){\c} \p(531,664){\c} \p(538,663){\c}
\p(545,663){\c} \p(552,662){\c} \p(559,662){\c} \p(565,661){\c}
\p(572,661){\c} \p(579,660){\c} \p(586,660){\c} \p(593,660){\c}
\p(600,659){\c} \p(607,659){\c} \p(613,658){\c} \p(620,658){\c}
\p(627,657){\c} \p(634,657){\c} \p(641,657){\c} \p(648,656){\c}
\p(654,656){\c} \p(661,656){\c} \p(668,655){\c} \p(675,655){\c}
\p(682,654){\c} \p(689,654){\c} \p(696,654){\c} \p(702,653){\c}
\p(709,653){\c} \p(716,653){\c} \p(723,652){\c} \p(730,652){\c}
\p(737,652){\c} \p(744,651){\c} \p(750,651){\c} \p(757,651){\c}
\p(764,650){\c} \p(771,650){\c} \p(778,650){\c} \p(785,649){\c}
\p(791,649){\c} \p(798,649){\c} \p(805,648){\c} \p(812,648){\c}
\p(819,648){\c} \p(826,647){\c} \p(833,647){\c} \p(839,647){\c}
\p(846,646){\c} \p(853,646){\c} \p(860,646){\c} \p(867,646){\c}
\p(874,645){\c} \p(881,645){\c} \p(887,645){\c} \p(894,644){\c}
\p(901,644){\c} \p(908,644){\c} \p(915,643){\c} \p(922,643){\c}
\p(928,643){\c} \p(935,643){\c} \p(942,642){\c} \p(949,642){\c}
\p(956,642){\c} \p(963,641){\c} \p(970,641){\c} \p(976,641){\c}
\p(983,641){\c} \p(990,640){\c} \p(997,640){\c} \p(1004,640){\c}
\p(1011,639){\c} \p(1018,639){\c} \p(1024,639){\c} \p(1031,639){\c}
\p(1038,638){\c} \p(1045,638){\c} \p(1052,638){\c} \p(1059,637){\c}
\p(1065,637){\c} \p(1072,637){\c} \p(1079,637){\c} \p(1086,636){\c}
\p(1093,636){\c} \p(1100,636){\c} \p(1107,635){\c} \p(1113,635){\c}
\p(1120,635){\c} \p(1127,635){\c} \p(1134,634){\c} \p(1141,634){\c}
\p(1148,634){\c} \p(1155,634){\c} \p(1161,633){\c} \p(1168,633){\c}
\p(1175,633){\c} \p(1182,632){\c} \p(1189,632){\c} \p(1196,632){\c}
\p(1202,632){\c} \p(1209,631){\c} \p(1216,631){\c} \p(1223,631){\c}
\p(1230,631){\c} \p(1237,630){\c} \p(1244,630){\c} \p(1250,630){\c}
\p(1257,629){\c} \p(1264,629){\c} \p(1271,629){\c} \p(1278,629){\c}
\p(1285,628){\c} \p(1292,628){\c} \p(1298,628){\c} \p(1305,628){\c}
\p(1312,627){\c} \p(1319,627){\c} \p(1326,627){\c} \p(1333,627){\c}
\p(1339,626){\c} \p(1346,626){\c} \p(1353,626){\c} \p(1360,625){\c}
\p(1257,459){\makebox(0,0)[r]{Conjugate Gradient Method}}
\p(1301,459){\D} \p(333,712){\D} \p(339,698){\D} \p(346,691){\D}
\p(353,685){\D} \p(360,680){\D} \p(367,675){\D} \p(374,671){\D}
\p(380,667){\D} \p(387,664){\D} \p(394,661){\D} \p(401,658){\D}
\p(408,654){\D} \p(415,649){\D} \p(422,646){\D} \p(428,643){\D}
\p(435,640){\D} \p(442,637){\D} \p(449,632){\D} \p(456,627){\D}
\p(463,624){\D} \p(470,618){\D} \p(476,615){\D} \p(483,606){\D}
\p(490,600){\D} \p(497,596){\D} \p(504,576){\D} \p(511,554){\D}
\p(517,546){\D} \p(524,541){\D} \p(531,536){\D} \p(538,533){\D}
\p(545,530){\D} \p(552,526){\D} \p(559,523){\D} \p(565,520){\D}
\p(572,517){\D} \p(579,512){\D} \p(586,508){\D} \p(593,504){\D}
\p(600,502){\D} \p(607,498){\D} \p(613,495){\D} \p(620,490){\D}
\p(627,484){\D} \p(634,479){\D} \p(641,474){\D} \p(648,469){\D}
\p(654,459){\D} \p(661,451){\D} \p(668,444){\D} \p(675,428){\D}
\p(682,410){\D} \p(689,403){\D} \p(696,399){\D} \p(702,396){\D}
\p(709,393){\D} \p(716,391){\D} \p(723,385){\D} \p(730,383){\D}
\p(737,379){\D} \p(744,376){\D} \p(750,372){\D} \p(757,368){\D}
\p(764,364){\D} \p(771,362){\D} \p(778,359){\D} \p(785,354){\D}
\p(791,351){\D} \p(798,347){\D} \p(805,342){\D} \p(812,335){\D}
\p(819,329){\D} \p(826,325){\D} \p(833,320){\D} \p(839,314){\D}
\p(846,304){\D} \p(853,299){\D} \p(860,292){\D} \p(867,287){\D}
\p(874,287){\D} \p(881,286){\D} \p(887,284){\D} \p(894,284){\D}
\p(901,286){\D} \p(908,286){\D} \p(1257,414){\makebox(0,0)[r]{Newton's
Method}} \p(1301,414){\s} \p(333,712){\s} \p(339,573){\s} \p(346,269){\s}
\p(353,244){\s} \p(360,246){\s} \p(367,240){\s} \p(374,240){\s}
\p(380,235){\s}
\end{picture}
\vskip-.25in
\caption[]{\protect\small\smallfrak\label{fig:SOnconv}\ignorespaces
Maximization of $\lyapunov$ on $\SO(20)$ (dimension $\SO(20)=190$),
where $N=\diag(20,\ldots,1)$.  The $i$th iterate is $H_i=\Theta_i^\T
Q\Theta_i$, $D_i$ is the diagonal matrix of eigenvalues of~$H_i$,
$H_0$ is near $N$, and $\|\cdot\|$ is the norm induced by the standard
inner product on~$\gl(n)$.  Geodesics and parallel translation were
computed using the algorithm of Ward and
Gray~\protect\cite{WG:1,WG:2}; the step sizes for the method of
steepest descent and the conjugate gradient method were computed using
Brockett's estimate~\protect\cite{Brockett:grad}.\parfillskip=0pt}
\endgroup  
\end{figure}
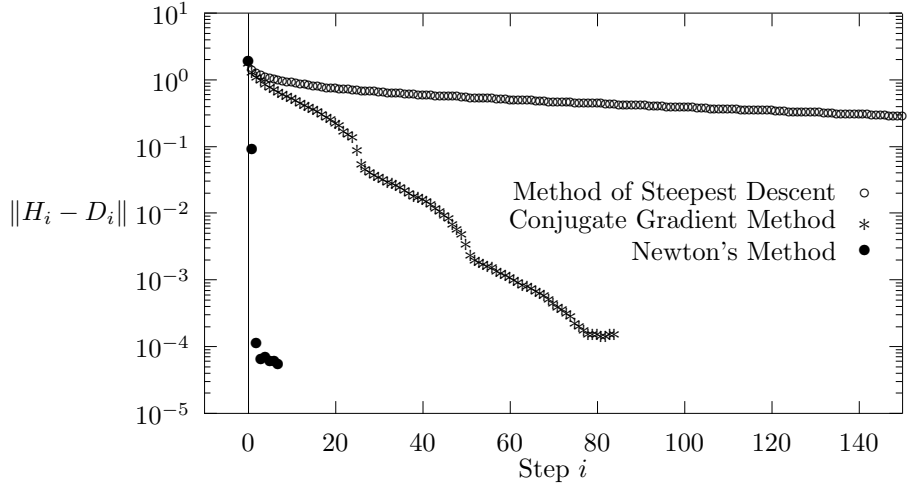

\begin{example}[The function $\lyapunov$]\label{eg:trHNcg}\ignorespaces
Let $\Theta$, $Q$, and $H$ be as in Examples \ref{eg:trHNgrad} and
\ref{eg:trHNnewton}.  As before, the natural Riemannian structure
of~$\SO(n)$ is used.  Let $X$, $Y\in\so(n)$.  The parallel translation
of~$Y$ along the geodesic $e^{tX}$ is given by the formula $\tau
Y=L_{e^{tX}{*}}e^{-(t/2)X}Ye^{(t/2)X}$, where $L_g$ denotes left
translation by~$g$.  Brockett's estimate (n.b.\
Eq.~(\ref{eq:Brockettest})) for the step size may be used in Step~1 of
Algorithm~\ref{al:cgman}.  The results of a numerical experiment
demonstrating the convergence of the conjugate gradient method
in~$\SO(20)$ are shown in Figure~\ref{fig:SOnconv}.
\end{example}

\begin{acknowledgments} The author enthusiastically thanks Tony Bloch
and the Fields Institute for the invitation to speak at the Fields
Institute and for their generous support during his visit.  The author
also thanks Roger Brockett for his suggestion to investigate conjugate
gradient methods on manifolds and for his criticism of this work, and
the referee for his helpful suggestions.  This work was supported in
part by the National Science Foundation under the Engineering Research
Center Program \hbox{NSF D CRD-8803012}, the Army Research Office
under Grant \hbox{DAA103-92-G-0164} supporting the Brown, Harvard, and
|MIT| Center for Intelligent Control, and by |DARPA| under Air Force
contract \hbox{F49620-92-J-0466}. \parfillskip=0pt
\end{acknowledgments}

\let\bibliographysize=\footnotesize  


\begin{thebibliography}{{\bibliographysize 99}}

\bibliographysize
\parskip=1pt plus 1pt minus1pt
\itemsep=0pt
\interlinepenalty=1000  
\hbadness=1500          
\references             


\bibitem{Bertsekas:newton}
|Bertsekas, D.~P.| Projected Newton methods for optimization problems
with simple constraints, !SIAM J.~Cont.~Opt.! \<20>:221--246, 1982.

\bibitem{Bertsekas}
\authorbar. !Constrained Optimization and Lagrange Multiplier Methods!.
New~York: Academic Press, 1982.

\bibitem{BBR1}
|Bloch, A.~M.|, |Brockett, R.~W.|, and |Ratiu, T.~S.| A new
formulation of the generalized Toda lattice equations and their fixed
point analysis via the momentum map, !Bull. Amer. Math. Soc.!
\<23>(2):477--485, 1990.

\bibitem{BBR2}
\authorbar. Completely integrable gradient flows, !Commun. Math. Phys.!
\<147>:57--74, 1992.

\bibitem{Botsaris:grad}
|Botsaris, C.~A.| Differential gradient methods, !J.~Math. Anal.
Appl.! \<63>:177--198, 1978.

\bibitem{Botsaris:class}
\authorbar. A class of differential descent methods for constrained
optimization, !J.~Math. Anal. Appl.! \<79>:96--112, 1981.

\bibitem{Botsaris:geod}
\authorbar. Constrained optimization along geodesics, !J.~Math. Anal.
Appl.! \<79>:295--306, 1981.

\bibitem{Brockett:match}
|Brockett, R.~W.| Least squares matching problems, !Lin. Alg. Appl.!
\def\r#1{{\rm#1}\penalty500\relax}\<122\r/123\r/124>:761--777, 1989.

\bibitem{Brockett:sort}
\authorbar. Dynamical systems that sort lists, diagonalize
matrices, and solve linear programming problems, !Lin. Alg. Appl.!
\<146>:79--91, 1991.

\bibitem{Brockett:grad}
\authorbar. Differential geometry and the design of gradient
algorithms, !Proc. Symp. Pure Math.! R. Green and S.~T. Yau, eds.
Providence, RI: Amer. Math. Soc., to appear.

\bibitem{Bcubed}
|Brown, A.~A.| and |Bartholomew-Biggs, M.~C.| Some effective methods
for unconstrained optimization based on the solution of systems of
ordinary differential equations, !J.~Optim. Theory Appl.!
\<62>(2):211--224, 1989.

\bibitem{Cheegin}
|Cheeger, J.| and |Ebin, D.~G.| !Comparison Theorems in Riemannian
Geometry!. Amsterdam: North-Holland Publishing Company, 1975.

\bibitem{Chu:sphere}
|Chu, M.~T.| Curves on $S^{n-1}$ that lead to eigenvalues or their
means of a matrix, !SIAM J.~Alg. Disc. Meth.! \<7>(3):425--432, 1986.

\bibitem{Chu:grad}
|Chu, M.~T.| and |Driessel, K.| The projected gradient method for
least squares matrix approximations with spectral constraints, !SIAM
J.~Numer. Anal.! \<27>(4):1050--1060, 1990.

\bibitem{Dunn:newton}
|Dunn, J.~C.| Newton's method and the Goldstein step length rule for
constrained minimization problems, !SIAM J.~Cont.~Opt.!
\<18>:659--674, 1980.

\bibitem{Dunn:grad}
\authorbar. Global and asymptotic convergence rate estimates for a
class of projected gradient processes, !SIAM J.~Cont.~Opt.!
\<19>:368--400, 1981.

\bibitem{Leonid}
|Faybusovich, L.| Hamiltonian structure of dynamical systems which
solve linear programming problems, !Phys.~D\/! \<53>:217--232, 1991.

\bibitem{Fletcher}
|Fletcher, R.| !Practical Methods of Optimization!, 2d~ed.  New~York:
Wiley \&~Sons, 1987.

\bibitem{FletcherReeves}
|Fletcher, R.| and |Reeves, C.~M.| Function minimization by conjugate
gradients, !Comput.~J.! \<7>(2):149--154, 1964.

\bibitem{FuhrLiu}
|Fuhrmann, D.~R.| and |Liu, B.| An iterative algorithm for locating
the minimal eigenvector of a symmetric matrix, !Proc. IEEE
ICASSP~84\/! pp.~45.8.1--4, 1984.

\bibitem{GillMurray}
|Gill, P.~E.| and |Murray, W.| Newton-type methods for linearly
constrained optimization, in !Numerical Methods for Constrained
Optimization!. P.~E. Gill and W. Murray, eds. London: Academic Press,
Inc., 1974.

\bibitem{GVL}
|Golub, G.~H.| and |Van~Loan, C.| !Matrix Computations!. Baltimore,
MD: Johns Hopkins University Press, 1983.

\bibitem{GandG}
|Golubitsky, M.| and |Guillemin, V.| !Stable Mappings and Their
Singularities!. New~York: Springer-Verlag, 1973.

\bibitem{Helgason}
|Helgason, S.| !Differential Geometry, Lie Groups, and Symmetric
Spaces!. New~York: Academic Press, 1978.

\bibitem{Uwe}
|Helmke, U.| Isospectral flows on symmetric matrices and the Riccati
equation, !Systems~{\sl\&} Control Lett.! \<16>:159--165, 1991.

\bibitem{HestenesStiefel}
|Hestenes, M.~R.| and |Stiefel, E.| Methods of conjugate gradients for
solving linear systems, !J.~Res. Nat. Bur. Stand.! \<49>:409--436,
1952.

\bibitem{HirschSmale}
|Hirsch, M.~W.| and |Smale, S.| On algorithms for solving $f(x)=0$,
!Comm. Pure Appl. Math.! \<32>:281--312, 1979.

\bibitem{Karcher}
|Karcher, H.| Riemannian center of mass and mollifier smoothing,
!Comm. Pure Appl. Math.! \<30>:509--541, 1977.

\bibitem{KobayshiandNomizu}
|Kobayashi,~S.| and |Nomizu,~K.| !Foundations of Differential
Geometry!, Vol.~2. New~York: Wiley Interscience Publishers, 1969.

\bibitem{Lagarias}
|Lagarias, J.~C.| Monotonicity properties of the Toda flow, the
QR-flow, and subspace iteration, !SIAM J.~Numer. Anal. Appl.!
\<12>(3):449--462, 1991.

\bibitem{Luenberger}
|Luenberger, D.~G.| !Introduction to Linear and Nonlinear
Programming!. Reading, MA: Addison-Wesley, 1973.

\bibitem{nineteendubious}
|Moler, C.| and |Van~Loan, C.| Nineteen dubious ways to compute the
exponential of a matrix, !SIAM Rev.! \<20>(4):801--836, 1978.

\bibitem{Nomizu}
|Nomizu, K.| Invariant affine connections on homogeneous spaces.
!Amer.~J.~Math.! \<76>:33--65, 1954.

\bibitem{Parlett}
|Parlett, B.| !The Symmetric Eigenvalue Problem!. Englewood Cliffs,
NJ: Prentice-Hall, 1980.

\bibitem{balreal}
|Perkins, J.~E.|, |Helmke, U.|, and |Moore, J.~B.| Balanced
realizations via gradient flow techniques, !Systems~{\sl\&} Control Lett.!
\<14>:369--380, 1990.

\bibitem{Polak}
|Polak, E.| !Computational Methods in Optimization!. New~York:
Academic Press, 1971.

\bibitem{Rudin:pma}
|Rudin, W.| !Principles of Mathematical Analysis!, 3d~ed. New~York:
McGraw-Hill, 1976.

\bibitem{Sargent}
|Sargent, R.~W.~H.| Reduced gradient and projection methods for
nonlinear programming, in !Numerical Methods for Constrained
Optimization!. P.~E. Gill and W. Murray, eds. London: Academic Press,
Inc., 1974.

\bibitem{Shub:ray}
|Shub, M.| Some remarks on dynamical systems and numerical analysis,
in !Dynamical Systems and Partial Differential Equations: Proc. VII
ELAM!. L.~Lara-Carrero and J.~Lewowicz, eds. Caracas: Equinoccio,
U.~Sim\'on Bol{\'\i}var, pp.~69--92, 1986.

\bibitem{ShubSmale:I}
|Shub, M.| and |Smale, S.| Computational complexity: On the geometry
of polynomials and a theory of cost, Part~I, !Ann. scient. \'Ec. Norm.
Sup.! \<4>(18):107--142, 1985.

\bibitem{ShubSmale:II}
\authorbar. Computational complexity: On the geometry of
polynomials and a theory of cost, Part~II, !SIAM J.~Comput.!
\<15>(1):145--161, 1986.

\bibitem{ShubSmale:conv}
\authorbar. On the existence of generally convergent
algorithms, !J.~Complex.! \<2>:2--11, 1986.

\bibitem{Smale:fta}
|Smale, S.| The fundamental theorem of algebra and computational
complexity,  !Bull. Amer. Math. Soc.! \<4>(1):1--36, 1981.

\bibitem{Smale:eff}
\authorbar. On the efficiency of algorithms in analysis,  !Bull.
Amer. Math. Soc.! \<13>(2):87--121, 1985.

\bibitem{Me}
|Smith, S.~T.| Dynamical systems that perform the singular value
decomposition, !Systems~{\sl\&} Control Lett.! \<16>:319--327, 1991.

\bibitem{Spivak}
|Spivak, M.| !A Comprehensive Introduction to Differential Geometry!,
2d~ed. Vols.~1, 2, 5, Houston, TX: Publish or Perish, Inc., 1979.

\bibitem{WG:1}
|Ward, R.~C.| and |Gray, L.~J.| Eigensystem computation for
skew-symmetric matrices and a class of symmetric matrices, !ACM Trans.
Math. Softw.! \<4>(3):278--285, 1978.

\bibitem{WG:2}
|Ward, R.~C.| and |Gray, L.~J.| Algorithm 530: An algorithm for computing
the eigensystem of skew-symmetric matrices and a class of symmetric
matrices, !ACM Trans. Math. Softw.! \<4>(3):286--289, 1978. See also
!Collected Algorithms from ACM!, Vol.~3. New~York: Assoc. Comput. Mach.,
1978.

\immediate\write16{Make the last page flush.}
\eject

\end{thebibliography}
\end{document}